\documentclass{elsarticle}
\usepackage[utf8]{inputenc}

\usepackage{verbatim, enumitem, amsmath, amsthm, amsfonts}
\usepackage{tikz}
\usetikzlibrary{arrows.meta}
\newtheorem{theorem}{Theorem}
\newtheorem{proposition}[theorem]{Proposition}
\newtheorem{lemma}[theorem]{Lemma}

\newtheorem{corollary}[theorem]{Corollary}

\theoremstyle{definition} 

\newtheorem{definition}[theorem]{Definition}

\newtheorem{example}[theorem]{Example}
\newtheorem{examples}[theorem]{Examples}
\newcommand{\spe}[1]{\mathbf{#1}}

\DeclareMathOperator{\Aut}{Aut}

\DeclareMathOperator{\Fix}{Fix}

\DeclareMathOperator{\id}{id}

\DeclareMathOperator{\Hilb}{F}
\DeclareMathOperator{\ps}{ps}
\DeclareMathOperator{\dcone}{dcone}

\DeclareMathOperator{\bdot}{\mathbf{\cdot}}

\DeclareMathOperator{\sgn}{sgn}
\DeclareMathOperator{\qsym}{QSYM}

\newcommand{\chifullevaluate}[6]{\Psi_{#1, #2}(#3,#4,#5;#6) }

\newcommand{\chiexevaluate}[2]{\chifullevaluate{\spe{#1}}{\varphi}{\spe{#2}}{\mathfrak{G}}{\mathbf{x}}{\mathfrak{g}}}

\newcommand{\chievaluate}{\chiexevaluate{H}{h}}

\begin{document}

\title{Chromatic Quasisymmetric Class Functions for combinatorial Hopf monoids}

\author{Jacob A. White}

\address{School of Mathematical and Statistical Sciences, University of Texas Rio Grande Valley, USA}


\begin{abstract}
    We study the chromatic quasisymmetric class function of a linearized combinatorial Hopf monoid. Given a linearized combinatorial Hopf monoid $H$, and an $H$-structure $h$ on a set $N$, there are proper colorings of $h$, generalizing graph colorings and poset partitions. We show that the automorphism group of $h$ acts on the set of proper colorings. The chromatic quasisymmetric class function enumerates the fixed points of this action, weighting each coloring with a monomial. For the Hopf monoid of graphs this invariant generalizes Stanley's chromatic symmetric function and specializes to the orbital chromatic polynomial of Cameron and Kayibi.

We also introduce the flag quasisymmetric class function of a balanced relative simplicial complex equipped with a group action.
We show that, under certain conditions, the chromatic quasisymmetric class function of $h$ is the flag quasisymmetric class function of a balanced relative simplicial complex that we call the coloring complex of $h$. We use this result to deduce various inequalities for the associated orbital polynomial invariants. We apply these results to several examples related to enumerating graph colorings, poset partitions, generic functions on matroids or generalized permutohedra, and others.

\end{abstract}

\maketitle

\section{Introduction}

Given a graph $G$, let $\mathfrak{G}$ be a subgroup of the automorphism group of $G$. Two colorings $f$ and $g$ are equivalent if $g = f \circ \mathfrak{g}^{-1}$ for some $\mathfrak{g} \in \mathfrak{G}$. Let $\chi(G,\mathfrak{G},k)$ denote the number of equivalence classes of $k$-colorings of $G$. Then $\chi(G,\mathfrak{G}, k)$ is a polynomial, called the orbital chromatic polynomial by Cameron and Kayibi \cite{cameron-kayibi}. 

In attempting to prove various properties about $\chi(G, \mathfrak{G}, k)$, and a symmetric function generalization, we discovered that the best approach was to apply tools from representation theory, which required working with class functions.
We now introduce the chromatic quasisymmetric class function of a graph $\spe{g}$. Given a graph $\spe{g}$, let $F(\spe{g})$ denote the set of all proper colorings $f:V(\spe{g}) \to \mathbb{N}$. If $\Aut(\spe{g})$ is the automorphism group of $\spe{g}$, then $\Aut(\spe{g})$ acts on $F(\spe{g})$ by $\mathfrak{g}f = f \circ \mathfrak{g}^{-1}$, where $f \in F(\spe{g})$ and $\mathfrak{g} \in \mathfrak{G}$.

Let $x_1, x_2, \ldots$ be commuting indeterminates, and let $\mathfrak{G} \subset \Aut(\spe{g})$. Given $\mathfrak{g} \in \mathfrak{G}$, we define \[\chi(\spe{g}, \mathfrak{G}, \mathbf{x}; \mathfrak{g}) = \sum_{f: \mathfrak{g}f = f} \prod_{v \in V(\spe{g})} x_{f(v)}\] where the sum is over proper colorings $f$ that are fixed by $\mathfrak{g}$.
The resulting power series is a quasisymmetric function, and as we vary $\mathfrak{g}$, we obtain a function on $\mathfrak{G}$ that is constant on conjugacy classes, and whose values are quasisymmetric functions. We call such functions \emph{quasisymmetric class functions}, and call $\chi(\spe{g}, \mathfrak{G}, \mathbf{x})$ the \emph{chromatic quasisymmetric class function} of $(\spe{g}, \mathfrak{G})$. In fact, $\chi(\spe{g}, \mathfrak{G}, \mathbf{x})$ is a symmetric class function. However, many of our other examples are not symmetric. Given two $\mathfrak{G}$-modules $V$ and $W$  and their respective characters $\rho_V$ and $\rho_W$, we write $\rho_V \leq_{\mathfrak{G}} \rho_W$ if $V$ is isomorphic to a submodule of $W$. Likewise, given a sequence $s_1 < s_2 < \cdots < d$, we let $\alpha(\{s_1, \ldots, s_k \}) = (s_1, s_2-s_1, \ldots, s_k-s_{k-1}, d-s_k)$.
We prove the following result:

\begin{theorem}
Let $\spe{g}$ be a graph on a vertex set $N$, and let $\mathfrak{G} \subseteq \Aut(\spe{g}).$ Then we have the following:
\begin{enumerate}
    \item Write \[\chi(\spe{g}, \mathfrak{G}, \mathbf{x}) = \sum_{S \subseteq [|N|-1]} c_{\mathfrak{G}, S} M_{\alpha(S)},\] where the $c_{\mathfrak{G}, S}$ are characters, and $M_{\alpha}$ is a monomial quasisymmetric function. For $S \subseteq T \subseteq [|N|-1]$, and $i \in [k]$, we have $c_{\mathfrak{G}, S} \leq_{\mathfrak{G}} c_{\mathfrak{G}, T}$.
     \item If we write $\chi(\spe{g}, \mathfrak{G}, x) = \sum_{i=0}^{|N|} f_i \binom{x}{i}$, then we have the following inequalities:
     \begin{enumerate}
    \item For $i \leq \frac{d}{2}$, we have $f_i \leq f_{i+1}$.
    \item For $i \leq \frac{d}{2}$, we have $f_i \leq f_{|N|-i}$.
    \item For all $i$, we have $(|N|-i)f_i \leq if_{i+1}$.
     \end{enumerate}
\end{enumerate}
\label{thm:maingraph}
\end{theorem}
A quasisymmetric class function which satisfies the first result is \emph{$M$-increasing}, while a polynomial which satisfies the first two conditions in the second result is \emph{strongly flawless}. 

The goal of this work is to provide a general method for creating new quasisymmetric class functions associated to combinatorial objects. An orbital version of the order polynomial of a poset was studied by Jochemko \cite{jochemko}, and a quasisymmetric function generalization for double posets was studied by Grinberg \cite{grinberg}. We generalized further to a quasisymmetric class function associated to a double poset \cite{white-2}, where we also proved an analogue of Theorem \ref{thm:maingraph}. Hence, we are after a general framework for constructing such quasisymmetric functions, and proving theorems about them. 

Previously, Aguiar, Bergeron, and Sottile \cite{aguiar-bergeron-sottile} have shown that the Hopf algebra of quasisymmetric functions $\qsym$ is the terminal object in the category of combinatorial Hopf algebras. Thus, given any collection of combinatorial objects that form the basis of a Hopf algebra, and a character, we obtain a quasisymmetric function invariant associated to each combinatorial object. We \cite{white-1} showed that, for Hopf monoids in species, we can describe the resulting quasisymmetric functions of Aguiar, Bergeron and Sottile as a generating function over $\varphi$-proper colorings. We review this construction in Subsection \ref{subsec:classfun}, and given the informal intuition here. A linearized Hopf monoid in species involves collections of labeled combinatorial objects $\spe{H}[N]$, which we call $\spe{H}$-structures, along with rules $\mu_{M,N}: \spe{H}[M] \times \spe{H}[N] \to \spe{H}[M \sqcup N]$ and $\Delta_{M,N}: \spe{H}[M \sqcup N] \to (\spe{H}[M] \times \spe{H}[N]) \cup \{0\}$ for how to combine or decompose combinatorial objects on disjoint label sets $M$ and $N$. We interpret $\Delta_{M,N}(\spe{h}) = 0$ to mean that $\spe{h}$ cannot be decomposed with respect to $M$ and $N$. This results in a more general notion of `linearized' Hopf monoid in species than what is studied in \cite{aguiar-mahajan-2}, but also allows the Hopf monoid of posets $\spe{P}$ as an example.

The character $\varphi$ determines which combinatorial objects are $1$-colorable. Then the $\varphi$-proper colorings are functions $f: N \to \mathbb{N}$ such that decomposing $\spe{h}$ with respect to $f^{-1}(0), f^{-1}(1), \ldots$, results in a sequence $\spe{h}_0$, $\spe{h}_1, \ldots$, of $1$-colorable $\spe{H}$-structures.

For example, for the Hopf monoid $\spe{G}$ of graphs, the map $\Delta_{M,N}$ corresponds to taking induced subgraphs, and the $1$-colorable graphs are independent sets, so for a graph $\spe{g}$ on a vertex set $N$, and a function $f: N \to \mathbb{N}$, the graph $\spe{g}_i$ is the induced subgraph on $f^{-1}(i)$, and then $f$ is proper if and only if $\spe{g}_i$ is $1$-colorable (equivalently, an independent set) for all $i$.

Given a linearized Hopf monoid $\spe{H}$ and a $\spe{H}$-structure $\spe{h}$, we let $F_{\varphi}(\spe{h})$ denote the set of $\varphi$-proper colorings of $\spe{h}$. Since combinatorial species involve labeled combinatorial objects, they come equipped with symmetric group actions. In particular, there is a notion of automorphism group $\Aut(\spe{h})$ for a $\spe{H}$-structure $\spe{h}$. 
Given a subgroup $\mathfrak{G} \subseteq \Aut(\spe{h})$, we show that $\mathfrak{G}$ acts on $F_{\varphi}(\spe{h})$. For graphs, this corresponds to definition of equivalence for colorings already given.
 For $\mathfrak{g} \in \mathfrak{G}$, we define \[\Psi_{\spe{H}, \varphi}(\spe{h}, \mathfrak{G}, \mathbf{x}; \mathfrak{g}) = \sum_{f \in F_{\varphi}(\spe{h}): \mathfrak{g}f = f} \prod_{v \in N} x_{f(v)}.\]
The resulting power series is a quasisymmetric function, and as we vary $\mathfrak{g}$, we obtain a $\qsym$-valued class function on $\mathfrak{G}$. We call this the \emph{$\varphi$-chromatic quasisymmetric class function of $(\spe{h}, \mathfrak{G})$}. There are a few interesting specializations to polynomial, including the orbital chromatic polynomial $\Psi^O_{\spe{H}, \varphi}(\spe{h}, \mathfrak{G}, x)$, which counts the number of orbits of $\varphi$-proper colorings with largest color at most $x$. These two invariants are general enough to study counting problems coming from group actions on:
\begin{enumerate}
    \item proper colorings of a graph,
    \item $P$-partitions of a poset,
    \item $M$-generic functions of a matroid, as defined in \cite{billera-jia-reiner},
    \item weak and strong colorings of mixed graphs, as defined in \cite{beck-et-al-2},
    \item $P$-partitions of double posets, as defined in \cite{grinberg},
    \item colorings of hypergraphs in the sense of \cite{qsym-hypergraph},
    \item $P$-generic functions of a generalized permutohedron, as defined in \cite{aguiar-ardila},
    \item colorings of simplicial complexes in the sense of \cite{benedetti-et-al}.
\end{enumerate}
 All of these examples are covered in some detail in Section \ref{sec:main}. In each case, the $\varphi$-chromatic quasisymmetric class function is $M$-increasing, and an orbital chromatic polynomial is strongly flawless.

We generalize Theorem \ref{thm:maingraph} to all of the above examples by giving a uniform proof that relies only on properties of the character $\varphi$.
We review the definition of balanced convex character in Section \ref{sec:realization}. Most characters studied in the literature are balanced convex characters, including the characters that give rise to all the various colorings listed above. The main condition for a balanced convex character is that $1$-colorable $\spe{H}$-structures can only be decomposed into $1$-colorable $\spe{H}$-structures.
We are able to prove the following results.
\begin{theorem}
Let $\spe{H}$ be a linearized Hopf monoid in species with linearized character $\varphi$. Let $N$ be a finite set, $\spe{h} \in \spe{H}[N]$, and $\mathfrak{G} \subseteq \Aut(\spe{h})$. Suppose that $\varphi$ is a balanced convex character.
\begin{enumerate}
    \item Then $\Psi_{\spe{H}, \varphi}(\spe{h}, \mathfrak{G}, \mathbf{x})$ is $M$-increasing. 
    \item The orbital polynomial $\Psi^O_{\spe{H}, \varphi}(\spe{h}, \mathfrak{G},x)$ is strongly flawless.
\end{enumerate}
\label{thm:main}
\end{theorem}
There are other specializations of $\Psi_{\spe{H}, \varphi}(\spe{h}, \mathfrak{G}, \mathbf{x})$ that we do not mention in this introduction.
 After we have defined the $\varphi$-chromatic quasisymmetric functions in \ref{subsec:classfun}, and various specializations in Subsection \ref{subsec:orbital}, we state Theorem \ref{thm:bigmain} in Section \ref{sec:main}, which contains all the inequalities for all the invariants we introduce in this paper.

The proofs of our results involve a mix of algebraic and geometric techniques. In Section \ref{sec:prelim}, we discuss quasisymmetric class functions and polynomial class functions. This material was studied previously in \cite{white-2}. We show that the property of being $M$-increasing for a quasisymmetric class function leads to various inequalities for the $f$-vector of the principal specialization. The second part of Theorem \ref{thm:main} follows from these more general properties. Thus, the main challenge is to show that our class functions are $M$-increasing.

In \cite{white-1}, we studied a generalization of Steingr\'imsson's coloring complex of a graph. Given a balanced convex character $\varphi$, we showed that there exists a balanced relative simplicial complex $\Sigma_{\varphi}(\spe{h})$ such that \[\Psi_{\spe{H}, \varphi}(\spe{h}, \mathbf{x}) = \sum_{S \subseteq [d]} F_S(\Phi) M_{\alpha(S)}\] where $F_S(\Phi)$ is the flag $f$-vector of $\Phi$, $M_{\alpha}$ is a monomial quasisymmetric function, and given $s_1 < \cdots < s_k < d$, we have $\alpha(\{s_1, \ldots, s_k \}) = (s_1, s_2-s_1, \ldots, s_k - s_{k-1}, d-s_k)$.
In Section \ref{sec:geometry}, we discuss balanced relative simplicial complexes $\Phi$ and define the flag quasisymmetric \emph{class} function $F(\Phi, \mathfrak{G}, \mathbf{x})$ for $\Phi$ with respect to a group action. We show that $\Aut(\spe{h})$ acts on $\Sigma_{\varphi}(\spe{h})$, and that $\Psi_{\spe{H}, \varphi}(\spe{h}, \mathfrak{G}, \mathbf{x}) = F(\Sigma_{\varphi}(\spe{h}), \mathfrak{G}, \mathbf{x}).$ We show in Theorem \ref{thm:increasing} that $F(\Phi, \mathfrak{G}, \mathbf{x})$ is always $M$-increasing. Taken together, we obtain the first part of Theorem \ref{thm:main}.

\section{Preliminaries}
\label{sec:prelim}
Given a basis $B$ for a vector space $V$ over $\mathbb{C}$, and $\vec{\beta} \in B, \vec{v} \in V$, we let $[\vec{\beta}] \vec{v}$ denote the coefficient of $\vec{\beta}$ when we expand $\vec{v}$ in the basis $B$.

Let $\mathbf{x} = x_1, x_2, \ldots $ be a sequence of commuting indeterminates. Recall that an \emph{integer composition} $\alpha$ of a positive integer $n$ is a sequence $(\alpha_1, \ldots, \alpha_k)$ of positive integers such that $\alpha_1+\cdots + \alpha_k = n$. We write $\ell(\alpha) = k$, and $\alpha \models n$. Let $n \in \mathbb{N}$ and let $F(\mathbf{x}) \in \mathbb{C}[[\mathbf{x}]]$ be a homogeneous formal power series in $\mathbf{x}$, where the degree of every monomial in $F(\mathbf{x})$ is $n$. Then $F(\mathbf{x})$ is a \emph{quasisymmetric function} if it satisfies the following property:
for every $i_1 < i_2 < \cdots < i_k$, and every integer composition $\alpha \models n$ with $\ell(\alpha) = k$, we have $[\prod_{j=1}^k x_{i_j}^{\alpha_j}]F(\mathbf{x}) = [\prod_{j=1}^k x_j^{\alpha_j}]F(\mathbf{x})$. Often, we will define quasisymmetric functions that are generating functions over a collection of functions. Given a function $w: S \to \mathbb{N}$, we define 
\begin{equation}
    \mathbf{x}^w = \prod_{v \in S} x_{w(v)}.
    \label{eq:xf}
\end{equation} For example, the chromatic symmetric function of a graph $G$ is defined as $\chi(G, \mathbf{x}) = \sum\limits_{f:V \to \mathbb{N}} \mathbf{x}^f$ where the sum is over all proper colorings of $G$.

Given an integer composition $\alpha = (\alpha_1, \alpha_2, \ldots, \alpha_k)$ of $n$, we let \[M_{\alpha} = \sum\limits_{i_1 < \cdots < i_k} \prod_{j=1}^k x_{i_j}^{\alpha_j}.\] These are the \emph{monomial quasisymmetric functions}, which form a basis for the ring of quasisymmetric functions.

Our proofs rely a lot on working with set compositions, and quasisymmetric functions related to set compositions.
Given a finite set $N$, a \emph{set composition} is a sequence $(S_1, \ldots, S_k)$ of disjoint non-empty subsets whose union is $N$. We denote set compositions as $S_1|S_2|\cdots|S_k$, and refer to the sets $S_i$ as \emph{blocks}. We use $C \models N$ to denote that $C$ is a set composition of $N$, and let $\ell(C) = k$ be the length of the composition. Given $C$, the associated integer composition is \begin{equation} 
\alpha(C) = (|C_1|, |C_2|, \ldots, |C_k|).
\label{eq:alphac}
\end{equation}
We refer to $\alpha(C)$ as the \emph{type} of $C$. We partially order set compositions by refinement. We write $B \leq C$ if $C$ is a refinement of $B$.

\subsection{Group actions and class functions}
\label{subsec:groupaction}

 Given a group action $\mathfrak{G}$ on a set $X$, we let $X / \mathfrak{G}$ denote the set of orbits. For $x \in X$, $\mathfrak{G}_x$ is the stabilizer subgroup, and $\mathfrak{G}(x)$ is the orbit of $x$. Finally, for $\mathfrak{g} \in \mathfrak{G}$, we let $\Fix_{\mathfrak{g}}(X) = \{x \in X: \mathfrak{g}x = x \}$.

There is an action of $\mathfrak{S}_N$ on the collection of all set compositions of $N$. Given a permutation $\mathfrak{g} \in \mathfrak{S}_N$, and a set composition $C \models N$, we let  \[\mathfrak{g} C = \mathfrak{g}(C_1)| \mathfrak{g}(C_2) | \cdots | \mathfrak{g}(C_k)\]
where $\mathfrak{g}(C_i) = \{ \mathfrak{g}x: x \in C_i \}$.

We assume familiarity with the theory of complex representations of finite groups - see \cite{fulton-harris} for basic definitions. Recall that, given any group action of $\mathfrak{G}$ on a finite set $X$, there is a group action on $\mathbb{C}^X$ as well, which gives rise to a representation. The resulting representations are called \emph{permutation representations}. 
Let $R$ be a $\mathbb{C}$-algebra. Then an \emph{$R$-valued class function} is a function $\chi: \mathfrak{G} \to R$ such that, for every $\mathfrak{g}, \mathfrak{h} \in \mathfrak{G}$, and $\chi \in C(\mathfrak{G}, R)$, we have $\chi(\mathfrak{hg}\mathfrak{h}^{-1}) = \chi(\mathfrak{g}).$ Let $C(\mathfrak{G}, R)$ be the set of $R$-valued class functions from $\mathfrak{G}$ to $R$. For our paper, $R$ is usually $\qsym$ or $\mathbb{C}[x]$. We refer to $\chi \in C(\mathfrak{G}, \mathbb{C})$ as class functions when no confusion arises.

There is an orthonormal basis of $C(\mathfrak{G}, \mathbb{C})$ given by the characters of the irreducible representations of $\mathfrak{G}$. We refer to elements $\chi \in C(\mathfrak{G}, \mathbb{C})$ that are integer combinations of irreducible characters as \emph{virtual characters}, and elements that are nonnegative integer linear combinations as \emph{effective characters}. Let $E(\mathfrak{G}, \mathbb{C})$ be the set of effective characters. Finally, we say $\chi$ is a \emph{permutation character} if it is the character of a permutation representation. We partially order $E(\mathfrak{G}, \mathbb{C})$ by saying $\chi \leq_{\mathfrak{G}} \psi$ if $\psi - \chi$ is an effective character.

Let $\mathbf{B}$ be a basis for $R$. For $b \in \mathbf{B}$, $\mathfrak{g} \in \mathfrak{G}$, and $\chi \in C(\mathfrak{G}, R)$, let $\chi_b(\mathfrak{g}) = [b]\chi(\mathfrak{g})$. Then $\chi_b$ is a $\mathfrak{C}$-valued class function. Thus we can write $\chi = \sum\limits_{b \in \mathbf{B}} \chi_b b$. Conversely, given a family $\chi_b$ of $\mathfrak{C}$-valued class functions, one for each $b \in \mathbf{B}$, the function $\chi$ defined by $\chi(\mathfrak{g}) = \sum\limits_{b \in \mathbf{B}} \chi_b(\mathfrak{g}) b$ is an $R$-valued class function in $C(\mathfrak{G}, R)$.

Let $\chi$ be an $R$-valued class function. We say that $\chi$ is \emph{$\textbf{B}$-realizable} if $\chi_b$ is a permutation character for all $b$.
If $\mathbf{B}$ has a partial order on it, then we say $\chi$ is \emph{$\mathbf{B}$-increasing} if for all $b \leq c$ in $\mathbf{B}$, we have $\chi_b \leq_{\mathfrak{G}} \chi_c$. Assuming $\chi$ is $\mathbf{B}$-realizable, this is equivalent to saying that $\chi_c$ is the character for the representation of $\mathfrak{G}$ on some module $V$, and $\chi_b$ is the character of a representation of a submodule of $V$.

A \emph{quasisymmetric class function} is a $\qsym$-valued class function. Given a quasisymmetric class function $F(\mathfrak{G}, \mathbf{x})$, we write $F(\mathfrak{G}, \mathbf{x}) = \sum_{\alpha \models n} f_{\mathfrak{G}, \alpha} M_{\alpha}$, where $f_{\mathfrak{G}, \alpha} \in C(\mathfrak{G}, \mathbb{C})$.

Finally, we define a function $\langle \cdot, \cdot \rangle: C(\mathfrak{G}, R) \times C(\mathfrak{G}, R) \to R$ by 
$\langle \chi, \psi \rangle = \frac{1}{|\mathfrak{G}|} \overline{\chi}(\mathfrak{g}) \psi(\mathfrak{g})$ where $\overline{x}$ is the complex conjugate. In the case where $R = \mathbb{C}$, this is the usual inner product on class functions.

\begin{proposition}
\label{prop:global} 
Let $\mathfrak{G}$ be a finite group, let $R$ be a $\mathbb{C}$-algebra with basis $\mathbf{B}$. Fix $\chi \in C(\mathfrak{G}, R)$.
\begin{enumerate}[label={(\arabic*)},itemindent=1em]
\item Given an irreducible character $\psi$, we have $\langle \chi, \psi \rangle = \sum\limits_{b, c \in B} \langle \chi_b, \psi_c \rangle b \cdot c$.
\item Suppose $\textbf{B}$ is partially ordered. Let $\psi \in C(\mathfrak{G}, \mathbb{C})$. If $\chi$ is $B$-increasing, then for all $b \leq c$ in $\textbf{B}$ we have $[b]\langle \psi, \chi \rangle \leq [c]\langle \psi, \chi \rangle.$  \label{prop:charincreasing}
\end{enumerate}
\end{proposition}

\subsection{Polynomial class functions and Principal specialization}
Given a polynomial $p(x)$ of degree $d$, define the $f$-vector $(f_0, \ldots, f_d)$ via $p(x) = \sum_{i=0}^d f_i \binom{x}{i}$. We say that $p(x)$ is \emph{strongly flawless} if the following inequalities are satisfied:
\begin{enumerate}
    \item for $0 \leq i \leq \frac{d-1}{2}$, we have $f_i \leq f_{i+1}$.
    \item For $0 \leq i \leq \frac{d}{2}$, we have $f_i \leq f_{d-i}$.
\end{enumerate}
There is a lot of interest in log-concave and unimodal sequences in combinatorics. We consider strongly flawless sequences to be interesting, as strongly flawless unimodal sequences can be seen as a generalization of symmetric unimodal sequences. Examples of results with strongly flawless sequences include the work of Hibi \cite{hibi} and Juhnke-Kubitzke and Van Le \cite{kubitzke}.
We denote the entries of the $f$-vector as $f_i(p(x))$ when discussing multiple polynomials.

There is a generalization of $f$-vector for polynomial class functions, which we call the \emph{equivariant $f$-vector}. Given a group $\mathfrak{G}$, and a polynomial class function $p(\mathfrak{G}, x)$, we write $p(\mathfrak{G}, x) = \sum_{i=1}^d f_i \binom{x}{i}$, where the $f_i$ are characters.  
We say that $p(\mathfrak{G}, x)$ is \emph{effectively flawless} if we have the following system of inequalities:
\begin{enumerate}
    \item For $0 \leq i \leq \frac{d-1}{2}$, we have $f_i \leq_{\mathfrak{G}} f_{i+1}$.
    \item For $0 \leq i \leq \frac{d}{2}$, we have $f_i \leq_{\mathfrak{G}} f_{d-i}$.
\end{enumerate}

Given a quasisymmetric function $F(\mathbf{x})$ of degree $d$, there is an associated polynomial $\ps(F)(x)$ given by principal specialization. For $x \in \mathbb{N}$, we set \[x_i = \begin{cases} 1 & i \leq x \\ 0 & i > x
\end{cases} \]
The resulting sequence is a polynomial function in $x$ of degree $d$, which we denote by $\ps(F)(x)$. 
If we write $F(\mathbf{x}) = \sum\limits_{\alpha \models d} c_{\alpha} M_{\alpha}$, then $f_i(\ps(F(\mathbf{x}))) = \sum\limits_{\alpha \models d: \ell(\alpha) = i} c_{\alpha}$.

Let $F(\mathfrak{G}, \mathbf{x})$ be a quasisymmetric class function of degree $d$, and $\mathfrak{g} \in \mathfrak{G}$. Define $\ps(F) \in C(\mathfrak{G}, \mathbb{C}[x])$ by $\ps(F)(x;\mathfrak{g}) = \ps(F(\mathbf{x}; \mathfrak{g}))(x).$ We refer to $\ps(F)$ as the principal specialization, which results in an polynomial class function.

The following results were obtained in \cite{white-2}.
\begin{proposition}
\label{prop:global2}
Let $F(\mathfrak{G}, \mathbf{x})$ be a quasisymmetric class function be of degree $d$, and $\mathfrak{g} \in \mathfrak{G}$.  
\begin{enumerate}[label={(\arabic*)},itemindent=1em]
    \item If we write $F(\mathfrak{G}, \mathbf{x}) = \sum\limits_{\alpha \models d} \chi_{\alpha} M_{\alpha}$, then $\ps(F) = \sum\limits_{i=0}^d \chi_{\alpha} \binom{x}{\ell(\alpha)}$.
    \item If $F(\mathfrak{G}, \mathbf{x})$ is $M$-realizable and $M$-increasing, then $\ps(F)$ is effectively flawless. \label{prop:ftohorbit}
\end{enumerate}
\end{proposition}

We have one new result to add.
\begin{proposition}
Let $F(\mathfrak{G}, \mathbf{x})$ be a quasisymmetric class function of degree $d$. Let $(f_0, \ldots, f_d)$ be the $f$-vector of $\ps(F)$. Then for all $i$, we have \[ (d-i)f_i \leq_{\mathfrak{G}} if_{i+1}.\]
\label{prop:global3}
\end{proposition}
\begin{proof}
We have 
\begin{align*} (d-i)f_i & = \sum_{\alpha \models d: \ell(\alpha) = i} (d-i) \chi_{\alpha} \\
& = \sum_{\alpha \leq \beta \models d: \ell(\alpha) = i, \ell(\beta) = i+1} \chi_{\alpha} \\ 
& \leq_{\mathfrak{G}} \sum_{\alpha \leq \beta \models d: \ell(\alpha) = i, \ell(\beta) = i+1} \chi_{\beta} \\
& = \sum_{\beta \models d: \ell(\beta) = i+1} i \chi_{\beta} \\
& = if_{i+1}.
\end{align*}
The second equality comes from observing that, given $\alpha \models d$ with $\ell(\alpha) = i$, there are exactly $(d-i)$ compositions $\beta \geq \alpha$ with $i+1$ parts. Similarly, the fourth equality comes from observing that, given $\beta \models d$ with $i+1$ parts, there are exactly $i$ compositions $\alpha$ with $i$ parts such that $\alpha \leq \beta$.
\end{proof}

\section{Linearized combinatorial Hopf monoids}
\label{sec:combhopfmon}

In this section, we review the definition of linearized combinatorial Hopf monoids. First, we define the notion of Hopf monoid in the category of linear species.
A \emph{linear species} is a functor $\spe{F}: Set \to Vec$ from the category of finite sets with bijections to the category of finite dimensional vector spaces over a field $\mathbb{K}$ and linear transformations.

A \emph{Hopf monoid} is a Hopf monoid object in the category of linear species \cite{aguiar-mahajan-1}. We refer to \cite{aguiar-mahajan-2,aguiar-ardila} for more details.
For every pair of disjoint finite sets $M, N$, there are multiplication maps $\mu_{M,N}: \spe{H}[M] \otimes \spe{H}[N] \to \spe{H}[M \sqcup N]$ and comultiplication maps $\Delta_{M,N}: \spe{H}[M \sqcup N] \to \spe{H}[M] \otimes \spe{H}[N]$. We focus only on \emph{connected} species, where $\dim \spe{H}[\emptyset] = 1$. 
We let $\spe{x} \cdot \spe{y} = \mu_{M,N}(\spe{x} \otimes \spe{y}) $.

A \emph{set species} is an endofunctor $\spe{F}: Set \to Set$.
Given a set species $\spe{F}$, there is an associated linear species $\mathbb{K}\spe{F}$ called the \emph{linearization}: we define $(\mathbb{K}\spe{F})_N$ to be the vector space with basis $\spe{F}_N$. We refer to $\spe{f}$ as an $\spe{F}$-structure if there exists a finite set $N$ such that $\spe{f} \in \spe{F}_N$. 

Let $\spe{H}$ be a species.
We say that $\spe{H}$ is a \emph{linearized} Hopf monoid if $\mathbb{K}\spe{H}$ is a Hopf monoid and:
\begin{enumerate}
    \item For every pair of disjoint finite sets $M, N$, the multiplication map $\mu_{M,N}$ on $\mathbb{K}\spe{H}$ is linearized from a function $\mu_{M,N}: \spe{H}[M] \times \spe{H}[N] \to \spe{H}[M \sqcup N]$. This means that for every $\spe{x} \in \spe{H}[M], \spe{y} \in \spe{H}[N]$, we have $\spe{x} \cdot \spe{y} \in \spe{H}[M \sqcup N]$.
    \item For every pair of disjoint finite sets $M, N$, the comultiplication map $\Delta_{M,N}$ on $\mathbb{K}\spe{H}$ is linearized from a function $\Delta_{M,N}: \spe{H}[M \sqcup N] \to (\spe{H}[M] \times \spe{H}[N]) \sqcup \{0\}$. Thus, for every $\spe{x} \in \spe{H}[M \sqcup N]$, if $\Delta_{M,N}(\spe{x}) \neq 0$ then there exists $\spe{x}|_M \in \spe{H}[M]$ and $\spe{x}/M \in \spe{H}[N]$ with $\Delta_{M,N}(\spe{x}) = \spe{x}|_M \otimes \spe{x}/M$.

\end{enumerate}
Our notion of linearized Hopf monoid is more general than the usual one found in the literature, because we allow $\Delta_{M,N}(\spe{h}) = 0$ for $\spe{h} \in \spe{H}[M \sqcup N]$. The advantage is that we view the species of posets as a linearized Hopf monoid, despite the fact that the coproduct is sometimes zero.

\begin{example}
Given a finite set $N$, let $\spe{E}_N = \{1\}$. This gives rise to the \emph{exponential species}.
\end{example}

\begin{example}

Given a finite set $N$, we let $\spe{C}_N$ denote the collection of set compositions with ground set $N$. This forms the species of set compositions. Given a set decomposition $N = S \sqcup T$, if $C \models S$ is given by $C_1|\cdots|C_k$ and $C' \models T$ is given by $C'_1|\cdots|C'_r$, then their product $C \cdot C'$ is the set composition $C_1|\cdots|C_k|C'_1|\cdots|C_r$. 
Given a set composition $C \models M \sqcup N$, we let $C|_M = C_1 \cap M | C_2 \cap M | \cdots | C_k \cap M$ and $C / M = C_1 \cap N | \cdots | C_k \cap N$, where it is understood that we remove any empty blocks from the composition. 
For example, $\Delta_{\{1,3,5\}, \{2,4,6\}}(12|35|46) = 1|35 \otimes 2|46$.

Then $\spe{C}$ is a linearized Hopf monoid.
\end{example}

\begin{example}
Given a finite set $N$, let $\spe{G}[N]$ denote the collection of graphs with vertex set $N$. Given a bijection $\sigma:M \to N$, and a graph $\spe{g} \in \spe{G}[M]$, define $\spe{G}[\sigma](\spe{g})$ to be the graph on $N$ with edges $ij$ if and only if $\sigma^{-1}(i)\sigma^{-1}(j)$ is an edge of $\spe{g}$. Then this gives rise to the set species of graphs $\spe{G}$, which is a linearized Hopf monoid. The product is given by $\spe{g} \cdot \spe{h} = \spe{g} \sqcup \spe{h}$, the disjoint union of graphs. Given a graph $\spe{g}$, and $S \subseteq N$, $\spe{g}|_S$ is the induced subgraph on $S$, and $\spe{g} / S$ is the induced subgraph on $N - S$. 
\end{example}

\begin{example}
Given a finite set $N$, let $\spe{P}[N]$ denote all partial orders on $N$. Given a bijection $\sigma:M \to N$, and a partial order $\spe{p} \in \spe{P}[M]$, define $\spe{P}[\sigma](\spe{g})$ to be the partial order on $N$ given by $x \leq y$ if and only if $\sigma^{-1}(x) \leq_{\spe{p}} \sigma^{-1}(y)$. Then this gives rise to the set species of posets $\spe{P}$ which is a linearized Hopf monoid. The product is given by $\spe{p} \cdot \spe{q} = \spe{p} \sqcup \spe{q}$, the disjoint union of partial orders.  Given a partial order $\spe{p}$, and $S \subseteq N$, we define $\spe{p}|_S$ to be the induced subposet on $S$. 
Given $N = S \sqcup T$, and $\spe{p} \in \spe{P}[N]$, we define 
\[ \Delta_{S, T}(\spe{p}) = 
\begin{cases}
(\spe{p}|_S, \spe{p}|_T) & \text{if $S$ is an order ideal of $\spe{p}$} \\
0 & \text{ otherwise.}
\end{cases}
\]
\end{example}

\begin{example}
Given a finite set $N$, let $\spe{M}[N]$ denote the collection of matroids with ground set $N$. Given a bijection $\sigma:M \to N$, and a matroid $\spe{m} \in \spe{M}[M]$, define $\spe{M}[\sigma](\spe{m})$ to be the matroid on $N$ where a set $S$ is a basis if and only if $\sigma^{-1}(S)$ is a basis of $\spe{m}$. This gives rise to the set species of matroids $\spe{M}$, which forms a linearized Hopf monoid. The product is given by the direct sum operation. Given a matroid $\spe{m}$, and $S \subset N$, we define $\spe{m}|_S$ to be the restriction, and $\spe{m} / S$ to be the contraction of matroids. 
\end{example}

A natural transformation $\varphi: \spe{H} \to \mathbb{K}\spe{E}$ of linear species, where $\spe{H}$ is a Hopf monoid in species, is a \emph{character} if for all disjoint finite sets $M$ and $N$, and all $x \in \spe{H}[M]$ and all $y \in \spe{H}[N]$, we have $\varphi_M(x) \cdot \varphi_N(y) = \varphi_{M \sqcup N}(x \cdot y)$. By an abuse of notation, we will write $\varphi(\spe{h})$ in place of $\varphi_{N(\spe{h})}(\spe{h})$, when no confusion will arise.

Suppose that $\spe{H}$ is a linearized Hopf monoid with a character $\varphi: \mathbb{K}\spe{H} \to \mathbb{K}\spe{E}$.
We say that $\varphi$ is a \emph{linearized} character if $\varphi_N(\spe{h}) \in \{0,1\}$ for all $\spe{H}$-structures $\spe{h}\in \spe{H}[N]$.
 A \emph{linearized combinatorial Hopf monoid} is a linearized Hopf monoid $\spe{H}$ with a \emph{linearized} character $\varphi$. 
The motivation is that many combinatorial Hopf algebras are studied where the character $\varphi$ only takes on the values $0$ and $1$.

First, we mention two examples of characters that are defined for every Hopf monoid.
\begin{example}
Let $\spe{H}$ be a Hopf monoid. We define $\zeta(\spe{h}) = 1$ for every $\spe{H}$-structure $\spe{h}$. We refer to $\zeta$ as the \emph{zeta character}.
\end{example}

\begin{example}
Let $\spe{H}$ be a Hopf monoid. We say a $\spe{H}$-structure $\spe{h} \in \spe{H}[N]$ is \emph{totally reducible} if $|N| = 1$, or there exists a nontrivial decomposition $N = S \sqcup T$, and totally reducible elements $\spe{x} \in \spe{H}[S]$ and $\spe{y} \in \spe{H}[T]$ such that $\spe{h} = \spe{x} \cdot \spe{y}$. We define 
    \[ \chi(\spe{h}) = 
\begin{cases}
1 & \text{if $\spe{h}$ is totally reducible} \\
0 & \text{ otherwise.}
\end{cases}
\]
We call $\chi$ the \emph{chromatic} character, and $(\spe{H}, \chi)$ is always a linearized combinatorial Hopf monoid. For instance, if we let $\spe{H} = \spe{G}$, then a graph $\spe{g}$ is totally reducible if and only if it is edgeless. If we let $\spe{H} = \spe{P}$, then a poset $\spe{p}$ is totally reducible if and only if it is an antichain. Finally, if $\spe{H} = \spe{M}$, then a matroid $\spe{m}$ is totally reducible if and only if is a direct sum of loops and coloops, which means $\spe{m}$ has a unique basis. These characters were studied in context of Hopf algebras in \cite{aguiar-bergeron-sottile}, and in the context of Hopf monoids in \cite{aguiar-ardila}.

\end{example}

\subsection{Chromatic quasisymmetric class functions}
\label{subsec:classfun}

We review the definition of the $\varphi$-chromatic quasisymmetric function.  Given a Hopf monoid $\spe{H}$, a finite set $N$, and a set composition $C = C_1|\cdots|C_k \models N$, we define $\Delta_C = (\id_{C_1} \otimes \Delta_{C/C_1}) \circ \Delta_{C_1, C_2 \cup \cdots \cup C_k }$. Given a character $\varphi$, we define $\varphi_C = \left(\prod_{i=1}^{|C|} \varphi_{C_i} \right) \circ \Delta_C$. Given $\spe{h} \in \spe{H}[N]$, we define 
\[\Psi_{\spe{H}, \varphi}(\spe{h}, \mathbf{x}) = \sum_{C \models N} \varphi_C(\spe{h}) M_{\alpha(C)}.\]
A similar quasiymmetric function invariant was defined for combinatorial Hopf algebras by Aguiar, Bergeron, and Sottile \cite{aguiar-bergeron-sottile}. Let $\spe{H} = \spe{G}$, and let $\spe{g}$ be a graph on $N$. Then $\Delta_C(\spe{g})$ breaks $\spe{g}$ up into induced subgraphs $\spe{g}_i$, one for each block $C_i$ of $C$. We see that $\varphi_{C_i}(\spe{g}_i) = 1$ if $\spe{g}_i$ has no edges, and is $0$ otherwise. Hence the term corresponding to $C$ is non-zero only when each block induces an independent set, in which case the term is $M_{\alpha(C)}$. In this manner, we see that $\Psi_{\spe{G}, \varphi}(\spe{g}, \mathbf{x})$ is the chromatic symmetric function introduced by Stanley \cite{stanley-chromatic}.

Stanley's original definition involved a weighted sum over colorings. We showed in \cite{white-1} that a similar formula exists for all linearized combinatorial Hopf monoids.
 Let $f: N \to \mathbb{N}$ be a function. There are only finitely many colors $i$ such that $f^{-1}(i) \neq \emptyset$: let $i_1 < \cdots < i_k$ be those colors. We let $N_i$ be the set of vertices $v$ such that $f(v) \leq i$.
We call $f$ a proper $\varphi$-coloring of $\spe{h}$ if $\varphi(\spe{h}|_{N_{i+1}}/N_i) = 1$ for all $i$. Let $F_{\varphi}(\spe{h})$ be the set of $\varphi$-proper colorings. We proved the following in \cite{white-1}:
\begin{proposition}
Let $\spe{H}$ be a linearized combinatorial Hopf monoid, $N$ be a finite set, and $\spe{h} \in \spe{H}[N]$. 
\begin{displaymath} \Psi_{\spe{H}, \varphi}(\spe{h}, \mathbf{x}) = \sum_{f \in F_{\varphi}(\spe{h})} \prod_{n \in N} x_{f(n)}. \end{displaymath}
\label{prop:oldcoloring}
\end{proposition}

We do not give the proof again, but we do discuss some of the main ideas. 
Let $f: N \to \mathbb{N}$. Let $i_1 < i_2 < \cdots < i_k$ be the natural numbers for which $f^{-1}(i_j) \neq \emptyset$. Define $C(f) = f^{-1}(i_1)|f^{-1}(i_2)|\cdots|f^{-1}(i_k)$. This is the composition associated with $f$. If we let $C(f) = C_1|\cdots|C_k$, then $\mathbf{x}^f = \prod_{j=1}^k x_{i_j}^{|C_j|}$. Conversely, given $i_1 < \cdots < i_k$, we see that $\prod_{j=1}^k x_{i_j}^{|C_j|}$ is the weight of the unique function defined by sending the vertices of $C_j$ to $i_j$. Given any composition $C \models N$, it follows that \[M_{\alpha(C)} = \sum_{f: C(f) = C} \mathbf{x}^f. \] 

We say set composition $C$ is $\varphi$-proper if $\varphi_C(\spe{h})=1$. We let $\mathcal{C}_{\varphi}(\spe{h})$ be the set of $\varphi$-proper set compositions. To prove Proposition \ref{prop:oldcoloring}, we showed that \[\Psi_{\spe{H}, \varphi}(\spe{h}, \mathbf{x}) = \sum_{C \in \mathcal{C}_{\varphi}(\spe{h})} M_{\alpha(C)}. \]
Then we showed that $f$ is $\varphi$-proper if and only if $C(f)$ is $\varphi$-proper. It then follows that \[\Psi_{\spe{H}, \varphi}(\spe{h}, \mathbf{x}) = \sum_{C \in \mathcal{C}_{\varphi}} M_{\alpha(C)} = \sum_{C \in \mathcal{C}_{\varphi}} \sum_{f: C(f)= C} \mathbf{x}^f = \sum_{f \in F_{\varphi}(\spe{h})} \mathbf{x}^f. \]

\begin{examples}
Let $\spe{G}$ be the Hopf monoid of graphs with character $\chi$. Then, as we noted, given a graph $\spe{g}$, $f \in F_{\varphi}$ is a proper coloring of $\spe{g}$, and $\Psi_{\varphi, \chi}(\spe{g}, \mathbf{x})$ is the chromatic symmetric function.

Let $\spe{P}$ be the Hopf monoid of posets, with character $\zeta$. Given a poset $\spe{p}$, a function $f \in F_{\zeta}$ is an order-preserving map $f: \spe{p} \to \mathbb{N}$, also known as a $\spe{p}$-partition, and $\Psi_{\spe{P}, \zeta}(\spe{p}, \mathbf{x})$ is a quasisymmetric function which enumerates $P$-partitions, originally introduced by Gessel \cite{gessel}.

Finally, let $\spe{M}$ be the Hopf monoid of matroids, with character $\chi$. Given a matroid $\spe{m}$, Billera, Jia, and Reiner showed that $f \in F_{\chi}$ if and only if $f$ is minimized on a unique basis of $\spe{m}$. These functions are called $\spe{m}$-generic functions, and the quasisymmetric function $\Psi_{\spe{M}, \chi}(\spe{m}, \mathbf{x})$ enumerating the $\spe{m}$-generic functions was first introduced by Billera, Jia, and Reiner \cite{billera-jia-reiner}. 
\end{examples}

Now we define automorphism groups.
Let $\spe{H}$ be a linearized combinatorial Hopf monoid with character $\zeta$. Let $N$ be a finite set, and for $\mathfrak{g} \in \mathfrak{S}_N$ and $\spe{h} \in \spe{H}[N]$, let $\mathfrak{g} \cdot \spe{h} = \spe{H}[\mathfrak{g}](\spe{h})$. This defines a group action of $\mathfrak{S}_N$ on $\spe{H}[N].$ Given $\spe{h} \in \spe{H}[N]$, we say $\mathfrak{g}$ is an \emph{automorphism} of $\spe{h}$ if $\mathfrak{g}\cdot \spe{h} = \spe{h}$. Let $\Aut(\spe{h})$ denote the set of automorphisms of $\spe{h}$, which is a subgroup of $\mathfrak{S}_N$. 

As an example, let $\spe{C}$ be the species of set compositions. Let $C = C_1|\cdots|C_k$ be a set composition on a finite set $N$. Given a permutation $\mathfrak{g} \in \mathfrak{S}_N$, we see that $\mathfrak{g}C = \mathfrak{g}(C_1)| \mathfrak{g}(C_2) | \cdots | \mathfrak{g}(C_k).$ Then $\Aut(C) = \{\mathfrak{g}: \mathfrak{g}C = C \} = \mathfrak{S}_{C_1} \times \mathfrak{S}_{C_2} \times \cdots \times \mathfrak{S}_{C_k}$.

Now we introduce our new quasisymmetric class function invariants.
Fix a finite set $N$. 
Let $\spe{h} \in \spe{H}[N]$ and let $\mathfrak{G}$ be a subgroup of $\Aut(\spe{h})$.
For $\mathfrak{g} \in \mathfrak{G}$, define \[\chiexevaluate{H}{\spe{h}} = \sum_{f \in \Fix_{\mathfrak{g}}(F_{\varphi}(\spe{h}))} \prod_{v \in N} x_{f(v)}.\]
This is the \emph{$\varphi$-chromatic quasisymmetric class function} associated to $\spe{H}$ with respect to $\spe{h}$. In Theorem \ref{thm:coloring}, we show that it is a quasisymmetric class function.

Naturally, there is a \emph{$\varphi$-chromatic polynomial class function} given by letting $\Psi_{\spe{H}, \varphi}(\spe{h}, \mathfrak{G}, x)$ denote the number of functions $f \in \Fix_{\mathfrak{g}}(F_{\varphi}(\spe{h}))$ with $f(N) \subseteq [x]$.

\begin{example}
Consider the linearized combinatorial Hopf monoid $(\spe{P}, \zeta)$, and let $\spe{p}$ be a partial order on a finite set $N$, and let $\mathfrak{G} \subseteq \Aut(\spe{p})$. Then $\mathfrak{G}$ acts on the set of $\spe{p}$-partitions. Moreover, for $\mathfrak{g} \in \mathfrak{G}$, we see that \[\Psi_{\spe{P}, \zeta}(\spe{p}, \mathfrak{G}, \mathbf{x}; \mathfrak{g}) = \sum_{\sigma: \mathfrak{g}\sigma = \sigma} \mathbf{x}^{\sigma} \]
where the sum is over $\spe{p}$-partitions that are fixed by $\mathfrak{g}$. 

Let $\spe{p}$ be the poset in Figure \ref{fig:poset}. Let $\mathbb{Z}/2\mathbb{Z}$ act on $\spe{p}$ by swapping $a$ with $c$ and $b$ with $d$. We let $\rho$ denote the regular representation.
Then \begin{align*}\Psi_{\spe{P}, \zeta}(\spe{p}, \mathbb{Z}/2\mathbb{Z}, \mathbf{x}) = M_4 & + \rho (M_{3,1}+M_{1,3})+M_{2,2} \\ & +\rho(M_{1,1,2}+2M_{1,2,1}+M_{2,1,1}+2 M_{1,1,1,1}).\end{align*}
\end{example}

\begin{figure}
\begin{center}

\begin{tikzpicture}
  \node[circle, draw=black, fill=white] (b) at (0,2) {$a$};
  \node[circle, draw=black, fill=white] (a) at (0,0) {$b$};
  \node[circle, draw=black, fill=white] (c) at (2,0) {$d$};
  \node[circle, draw=black, fill=white] (d) at (2,2) {$c$};
 \draw[-Latex] (b) -- (a);
  \draw[-Latex] (d) edge (c);
  \draw[-Latex] (b) -- (c);
    \draw[-Latex] (d) -- (a);

\end{tikzpicture}

\end{center}
\caption{A poset.}
\label{fig:poset}
\end{figure}

\begin{example}
Consider the linearized combinatorial Hopf monoid $(\spe{G}, \chi)$, and let $\spe{g}$ be a graph on a finite set $N$, and fix $\mathfrak{G} \subseteq \Aut(\spe{g})$. Then $\Psi_{\spe{G}, \chi}(\spe{g}, \mathfrak{G}, \mathbf{x})$ is the same quasisymmetric class function mentioned in the introduction in Theorem \ref{thm:maingraph}.

Let $\spe{g}$ be the graph in Figure \ref{fig:graph}. Let $\mathbb{Z}/4\mathbb{Z}$ act on $\spe{g}$ by cyclic rotation. We let $\rho$ denote the regular representation, and $\sgn$ denote the sign representation.
Then \[\Psi_{\spe{G}, \chi}(\spe{g}, \mathbb{Z}/4\mathbb{Z}, \mathbf{x}) =  (1+\sgn)M_{2,2}+\rho(M_{1,1,2}+M_{1,2,1}+M_{2,1,1}+6 M_{1,1,1,1}).\]
\end{example}

\begin{example}
Consider the linearized combinatorial Hopf monoid $(\spe{M}, \chi)$, and let $\spe{m}$ be a matroid on a finite set $N$, and fix $\mathfrak{G} \subseteq \Aut(\spe{m})$. For $\mathfrak{g} \in \mathfrak{G}$, we have \[\Psi_{\spe{M}, \chi}(\spe{m}, \mathfrak{G}, \mathbf{x}; \mathfrak{g}) = \sum_{f: \mathfrak{g}f = f} \mathbf{x}^f \]
where the sum is over $\spe{m}$-generic functions that are fixed by $\mathfrak{g}$.
We observe that $\Psi_{\spe{M}, \chi}(\spe{m}, \{e\}, \mathbf{x})$ is the Billera-Jia-Reiner quasisymmetric function associated to a matroid.

Let $\spe{m}$ be the uniform matroid on four elements of rank two. That is, $\spe{m}$ have vertices $\{0, 1, 2, 3 \}$, and every subset of size two is a basis. Let $\mathbb{Z}/4\mathbb{Z}$ act on $\spe{m}$ by cyclic rotation of the vertices. We let $\rho$ denote the regular representation, and $\sgn$ denote the sign representation.
Then \[\Psi_{\spe{M}, \chi}(\spe{m}, \mathbb{Z}/4\mathbb{Z}, \mathbf{x}) =  (\rho+1+\sgn)M_{2,2}+3\rho(M_{1,1,2}+M_{2,1,1}+6 M_{1,1,1,1}).\]

\end{example}

We now describe $\Psi_{\spe{H}, \varphi}(\spe{h}, \mathfrak{G}, \mathbf{x})$ in terms of the basis of monomial quasisymmetric functions.
  We say a set composition is $\varphi$-proper if $\varphi_C(\spe{h}) = 1$. Let $\mathcal{C}_{\varphi}(\spe{h})$ be the set of $\varphi$-proper set compositions, and note that $\mathcal{C}_{\varphi}(\spe{h})$ is a $\mathfrak{G}$-set. 

Given $\alpha \models |N|$, define \[\mathcal{C}_{\spe{h}, \alpha} = \{C: C \models N, \alpha(C) = \alpha, \varphi_C(\spe{h}) = 1 \}.\] Given $\mathfrak{g} \in \mathfrak{G}$, and $C \in \mathcal{C}_{\spe{h}, \alpha}$, we have $\mathfrak{g} \cdot C \in \mathcal{C}_{\spe{h}, \alpha}$. Hence $\mathfrak{G}$ acts on $\mathcal{C}_{\spe{h}, \alpha}$. If we take the span of $\mathcal{C}_{\spe{h}, \alpha}$ over $\mathbb{K}$ we get a $\mathfrak{G}$-module $V_{\spe{h}, \alpha}$. We let $\Psi_{\spe{h}, \alpha}$ denote the resulting permutation character.
\begin{figure}
\begin{center}
\begin{tikzpicture}
  \node[circle, draw=black, fill=white] (b) at (0,2) {$a$};
  \node[circle, draw=black, fill=white] (a) at (0,0) {$b$};
  \node[circle, draw=black, fill=white] (c) at (2,0) {$d$};
  \node[circle, draw=black, fill=white] (d) at (2,2) {$c$};
 \draw[-] (b) -- (a);
  \draw[-] (d) -- (c);
  \draw[-] (b) -- (d);
    \draw[-] (c) -- (a);

\end{tikzpicture}

\caption{A graph.}
\end{center}
\label{fig:graph}
\end{figure}

\begin{theorem}
Let $\spe{H}$ be a linearized combinatorial Hopf monoid. Fix a finite set $N$, a $\spe{H}$-structure $\spe{h} \in \spe{H}[N]$, and a group $\mathfrak{G} \subseteq \Aut(\spe{h})$.
Then we have the following identities:
\begin{enumerate} 
\item \begin{equation}\chievaluate =  \sum_{C \in \mathcal{C}_{\varphi}(\spe{h}): \mathfrak{g}C = C} M_{\alpha(C)}  \label{eq:fundamental} \end{equation}

\item \begin{equation} \Psi_{\spe{H}, \varphi}(\spe{h}, \mathfrak{G}, \mathbf{x}) = \sum_{\alpha \models |N|} \Psi_{\spe{h}, \alpha} M_{\alpha}. \end{equation}
\item \begin{equation}\Psi_{\spe{H}, \varphi}(\spe{h}, \mathfrak{G}, x) = \sum_{\alpha \models |N|} \Psi_{\spe{h}, \alpha}\binom{x}{|\alpha|}.\end{equation}
\end{enumerate}

\label{thm:coloring}
\end{theorem}
\begin{proof}
Fix $\spe{H}, N$, $\spe{h}$, and $\mathfrak{G}$ as in the statement of the theorem.

For the first formula, fix $\mathfrak{g} \in \mathfrak{G}$ and let $f \in \Fix_{\mathfrak{g}}(F_{\varphi}(\spe{h}))$. 
 We see that $\mathfrak{g}f = f$ if and only if for all $v \in N$, we have $f(\mathfrak{g}v) = f(v)$. This is equivalent to requiring $\mathfrak{g}f^{-1}(i_j) = f^{-1}(i_j)$ for all $j$, which means that $\mathfrak{g}C(f) = C(f)$.
 
Thus we obtain
\[ \sum\limits_{C \in \Fix_{\mathfrak{g}}(\mathcal{C}_{\varphi}(\spe{h}))} M_{\alpha(C)} = \sum\limits_{C \in \Fix_{\mathfrak{g}}(\mathcal{C}_{\varphi}(\spe{h}))} \sum\limits_{f: C(f) = f} \mathbf{x}^f =  \sum\limits_{f: C(f) \in \Fix_{\mathfrak{g}}(\mathcal{C}_{\varphi}(\spe{h}))} \mathbf{x}^f.\] 
The first equality uses the relation between functions an their compositions. We know that $\mathfrak{g}C(f) = C(f)$ if and only $\mathfrak{g}f = f$, and $C(f) \in \mathcal{C}_{\varphi}(\spe{h})$ if and only if $f \in F_{\varphi}(\spe{h})$. Thus we have \[\sum\limits_{C \in \Fix_{\mathfrak{g}}(\mathcal{C}_{\varphi}(\spe{h}))} M_{\alpha(C)} = \sum_{f \in \Fix_{\mathfrak{g}}(F_{\varphi}(\spe{h}))} \mathbf{x}^f = \Psi_{\spe{H}, \varphi}(\spe{h}, \mathfrak{G}, \mathbf{x}; \mathfrak{g}).\]

To prove the second identity, we see that
\[ \sum\limits_{C \in \Fix_{\mathfrak{g}}(\mathcal{C}_{\varphi}(\spe{h}))} M_{\alpha(C)} = \sum\limits_{\alpha \models |N|} \sum\limits_{C \in \Fix_{\mathfrak{g}}(\mathcal{C}_{\spe{h}, \alpha})} M_{\alpha} = \sum\limits_{\alpha \models |N|} \Psi_{\spe{h}, \alpha}(\mathfrak{g}) M_{\alpha}.\]

Now we prove the third identity. Let $x \in \mathbb{N}$, and $\mathfrak{g} \in \mathfrak{G}$. By definition, $\Psi_{\spe{H}, \varphi}(\spe{h}, \mathfrak{G}, x; \mathfrak{g})$ counts the number of $\varphi$-proper colorings $f$ with $f(N) \subseteq [x]$ and $\mathfrak{g}f = f$. However, we see that $\ps(\Psi_{\spe{H}, \varphi}(\spe{h}, \mathfrak{G}, \mathbf{x}); \mathfrak{g})(x)$ also counts this same set of colorings (because setting $x_i = 0$ for $i > x$ reduces the summation to only those colorings where $f(N) \subseteq [x]$, and setting $x_i = 1$ for $i \leq n$ counts each function with weight $1$).

Thus $\Psi_{\spe{H}, \varphi}(\spe{h}, \mathfrak{G}, x) = \ps(\Psi_{\spe{H}, \varphi}(\spe{h}, \mathfrak{G}, \mathbf{x}))$. 
The third formula follows immediately from principal specialization, and the fact that $\ps M_{\alpha} = \binom{x}{\ell(\alpha)}$.

\end{proof}

\subsection{Orbital Chromatic Quasisymmetric Function}
\label{subsec:orbital}

Now we discuss our orbital invariants. Fix a linearized combinatorial Hopf monoid $\spe{H}$, and a $\spe{H}$-structure $\spe{h}\in\spe{H}[N]$ where $N$ is a finite set. Finally, let $\mathfrak{G} \subseteq \Aut(\spe{h})$. Given two $\varphi$-proper colorings $f$ and $g$ of $\spe{h}$, we say $f$ and $g$ are $\mathfrak{G}$-equivalent if $f = g \circ \mathfrak{g}^{-1}$ for some $\mathfrak{g} \in \mathfrak{G}$. We see that if $f$ and $g$ are $\mathfrak{G}$-equivalent, then $\mathbf{x}^f = \mathbf{x}^g$. Let $F_{\varphi}/\mathfrak{G}$ denote the set of equivalence classes. For $C \in F_{\varphi}/\mathfrak{G}$, we define $\mathbf{x}^C = \mathbf{x}^f$ for any $f \in C$.

Then we define the \emph{orbital $\varphi$-chromatic quasisymmetric function} to be \[\Psi_{\spe{H}, \varphi}^O(\spe{h}, \mathfrak{G}, \mathbf{x}) = \sum_{C \in F_{\varphi}/\mathfrak{G}} \mathbf{x}^C. \]

Similarly, if we let $F_{\varphi, k}(\spe{h})$ be the set of $\varphi$-colorings $f$ such that $f(N) \subseteq [k]$, then $\mathfrak{G}$ acts on $F_{\varphi, k}(\spe{h})$, and we let $\Psi_{\spe{H}, \varphi}^O(\spe{h}, \mathfrak{G}, x) = |F_{\varphi, k}(\spe{h})/\mathfrak{G}|$. This is the \emph{orbital $\varphi$-chromatic polynomial}.

We observe that $\Psi_{\spe{P}, \zeta}^O(\spe{p}, \mathfrak{G}, x)$ is the orbital order polynomial introduced by Jochemko \cite{jochemko}.
We observe that $\Psi_{\spe{G}, \chi}^O(\spe{g}, \mathfrak{G}, x)$ is the orbital chromatic polynomial introduced by Cameron and Kayibi \cite{cameron-kayibi}.
For matroids, $\Psi_{\spe{M}, \chi}^O(\spe{m}, \mathfrak{G}, x)$ is a new invariant, forming an orbital version of the Billera-Jia-Reiner polynomial.

\begin{theorem}
Let $\spe{H}$ be a linearized combinatorial Hopf monoid. Fix a finite set $N$, a $\spe{H}$-structure $\spe{h} \in \spe{H}[N]$, and a group $\mathfrak{G} \subseteq \Aut(\spe{h})$.
Then we have the following identities:
\begin{enumerate} 

\item \begin{equation} \Psi^O_{\spe{H}, \varphi}(\spe{h}, \mathfrak{G}, \mathbf{x}) = \frac{1}{|\mathfrak{G}|} \sum_{\mathfrak{g} \in \mathfrak{G}} \Psi_{\spe{H}, \varphi}(\spe{h},\mathfrak{G}, \mathbf{x};\mathfrak{g}) \end{equation}
\item \begin{equation} [M_{\alpha}]\Psi^O_{\spe{H}, \varphi}(\spe{h}, \mathfrak{G}, \mathbf{x}) = |\mathcal{C}_{\varphi, \alpha}(\spe{h})/\mathfrak{G}|. \end{equation}
\item \begin{equation}\Psi^O_{\spe{H}, \varphi}(\spe{h}, \mathfrak{G}, x) = \ps (\Psi^O_{\spe{H}, \varphi}(\spe{h},\mathfrak{G}, \mathbf{x})) = \frac{1}{|\mathfrak{G}|}\sum_{\mathfrak{g} \in \mathfrak{G}} \Psi_{\spe{H}, \varphi}(\spe{h},\mathfrak{G}, x;\mathfrak{g}).\end{equation}
\end{enumerate}

\label{thm:coloring2}
\end{theorem}
\begin{proof}
Fix $\spe{H}$, $N$, $\spe{h}$ and $\mathfrak{G}$ as in the statement of the Theorem. We first prove the first formula by comparing coefficients on both sides.
Let $i_1 < i_2 < \cdots < i_k$, and $\alpha \models |N|$. Let $F_{(i_1, \ldots, i_k), \alpha}(\spe{g})$ denote the set of $\varphi$-proper colorings $f:N \to \mathbb{N}$ such that 
\[|f^{-1}(i)| = \begin{cases} \alpha_j & i = i_j \mbox{ for some } j \\ 0 & \mbox{ otherwise} \end{cases} \]
Then $\mathfrak{G}$ acts on $F_{(i_1, \ldots, i_k), \alpha}(\spe{g})$, and
 \begin{align*} [\prod_{j=1}^k x_{i_j}^{\alpha_j}] \Psi_{\spe{H}, \varphi}^O(\spe{h}, \mathfrak{G}, \mathbf{x}) & = |F_{(i_1, \ldots, i_k), \alpha}(\spe{g})/\mathfrak{G}| \\ & = \frac{1}{|\mathfrak{G}|} \sum_{\mathfrak{g} \in \mathfrak{G}} |\Fix_{\mathfrak{g}}(F_{(i_1, \ldots, i_k), \alpha})| \\
 & = \frac{1}{|\mathfrak{G}|} \sum_{\mathfrak{g} \in \mathfrak{G}} [\prod_{j=1}^k x_{i_j}^{\alpha_j}] \Psi_{\spe{H}, \varphi}(\spe{h}, \mathfrak{G}, \mathbf{x}; \mathfrak{g}) \\
  & = [\prod_{j=1}^k x_{i_j}^{\alpha_j}] \frac{1}{|\mathfrak{G}|} \sum_{\mathfrak{g} \in \mathfrak{G}}  \Psi_{\spe{H}, \varphi}(\spe{h}, \mathfrak{G}, \mathbf{x}; \mathfrak{g})  \end{align*}
where the second equality uses Burnside's Lemma.

To prove the second formula, observe that
\begin{align*} [M_{\alpha}] \Psi_{\spe{H}, \varphi}^O(\spe{h}, \mathfrak{G}, \mathbf{x}) 
 & = \frac{1}{|\mathfrak{G}|} \sum_{\mathfrak{g} \in \mathfrak{G}} [M_\alpha] \Psi_{\spe{H}, \varphi}(\spe{h}, \mathfrak{G}, \mathbf{x}; \mathfrak{g}) \\
 & = \frac{1}{|\mathfrak{G}|} \sum_{\mathfrak{g} \in \mathfrak{G}} \Psi_{\spe{h}, \alpha}(\mathfrak{g}) \\
 & = |\mathcal{C}_{\varphi, \alpha}(\spe{h})/\mathfrak{G}| \end{align*}
where the first equality is the first formula, the second equality uses the second formula of Theorem \ref{thm:coloring}, and the third equality comes from Burnside's Lemma and the observation that $\Psi_{\spe{h}, \alpha}(\mathfrak{g}) = |\Fix_{\mathfrak{g}}(\mathcal{C}_{\varphi, \alpha}(\spe{h}))|$.

The first equality for the third identity follows from the definition of principal specialization. The second equality can be proven in a similar manner to our first formula.

\end{proof}

\section{Main Theorem and Examples}
\label{sec:main}
We now give the definition of balanced convex character, and the full version of Theorem \ref{thm:main}.
\begin{definition}
We say that $\varphi$ is a \emph{balanced convex character} if, for every finite set $N$ and ever $\spe{H}$-structure $\spe{h} \in \spe{H}[N]$, the following conditions are satisfied:
\begin{enumerate}
    \item If $|N| = 1$, then $\varphi(\spe{h}) = 1$.
    \item If $|N| > 1$, then there exists non-empty $S \subset N$ such that $\Delta_{S, N \setminus S}(\spe{h}) \neq 0$.
    \item If $\varphi(\spe{h}) = 1$, and there exists $S \subseteq N$ such that $\Delta_{S, N \setminus S}(\spe{h}) \neq 0$, then $\varphi(\spe{h}|_S) = \varphi(\spe{h}/S) = 1$.
\end{enumerate}
\label{def:convex}
\end{definition}

The following we proven in \cite{white-1}.
\begin{example}
Let $\spe{H}$ be a linearized Hopf monoid. Suppose that for all finite sets $|N| > 1$ and all $\spe{H}$-structures $\spe{h} \in \spe{H}[N]$, there exists non-empty $S \subset N$ such that $\Delta_{S, N \setminus S}(\spe{h}) \neq 0$. Then $\zeta$ and $\chi$ are balanced convex characters.

In particular, $\zeta$ and $\chi$ are balanced convex characters for the linearized Hopf monoids of graphs $\spe{G}$, posets $\spe{P}$, and matroids $\spe{M}$.
\label{thm:chiconvex}
\end{example}

Here is the full main result:
\begin{theorem}
Let $\spe{H}$ be a linearized combinatorial Hopf monoid with a balanced convex character $\varphi$. Let $\spe{h}$ be a $\spe{H}$-structure on a finite set $N$, and fix $\mathfrak{G} \subseteq \Aut(\spe{h})$.
\begin{enumerate}
    \item We have $\Psi_{\spe{H}, \varphi}(\spe{h}, \mathfrak{G}, \mathbf{x})$ and $\Psi^O_{\spe{H}, \varphi}(\spe{h}, \mathfrak{G}, \mathbf{x})$ are $M$-increasing.
    \item We have that $\Psi_{\spe{H}, \varphi}(\spe{h}, \mathfrak{G}, x)$ is effectively flawless. Moreover, for all $i$, we have \[(|N|-i)f_i(\Psi_{\spe{H}, \varphi}(\spe{h}, \mathfrak{G}, x)) \leq_{\mathfrak{G}} if_{i+1}(\Psi_{\spe{H}, \varphi}(\spe{h}, \mathfrak{G}, x)).\]
    \item We have that $\Psi^O_{\spe{H}, \varphi}(\spe{h}, \mathfrak{G}, x)$ is strongly flawless. Moreover, for all $i$, we have \[(|N|-i)f_i(\Psi^O_{\spe{H}, \varphi}(\spe{h}, \mathfrak{G}, x)) \leq if_{i+1}(\Psi^O_{\spe{H}, \varphi}(\spe{h}, \mathfrak{G}, x)).\]
\end{enumerate}
\label{thm:bigmain}
\end{theorem}
We prove Theorem \ref{thm:bigmain} after the proof of Theorem \ref{thm:autgeometry}, which is used to prove the first part. The rest of the proof relies mostly on the material from Section \ref{sec:prelim}, as well as Theorem \ref{thm:coloring} and Theorem \ref{thm:coloring2}.

Since $\chi$ and $\zeta$ are balanced convex characters for the linearized combinatorial Hopf monoids of graphs, posets, and matroids, we see that Theorem \ref{thm:bigmain} applies to all three cases. In particular, the quasisymmetric class functions which enumerate proper colorings of a graph $\spe{g}$, $\spe{p}$-partitions of a poset $\spe{p}$, or $\spe{m}$-generic functions of a matroid $\spe{m}$, each with respect to group actions, are $M$-increasing. The corresponding orbital polynomials are all strongly flawless.

We remark that the result for posets is already a special case for similar results we obtained previously for double posets \cite{white-2}. Similarly, the result for graphs can also be deduced from previous results we have for digraphs \cite{white-2}. However, the new proof is more general, and includes all of the new examples in the next subsections. For the sake of brevity, we refer to papers in the literature that describe the Hopf monoid structure for the remaining examples, and focus instead on the resulting quasisymmetric functions.

\subsection{Acyclic mixed graphs}
 Given a finite set $N$, a \emph{mixed graph} is a triple $(N, U, D)$, where $U$ is a set of undirected edges, and $D$ is a set of directed edges.  A mixed graph is \emph{acyclic} if it does not contain a directed cycle. If we let $\spe{MG}[N]$ denote the set of acyclic mixed graphs on $N$, we obtain a species. The Hopf monoid structure on $\spe{MG}$ was studied in \cite{white-1}.

There are two polynomial invariants associated to acyclic mixed graphs: the weak and strong chromatic polynomial, both introduced in \cite{beck-et-al-2}, motivated by work in \cite{beck-et-al}. Given an acyclic mixed graph $\spe{g},$ the \emph{weak chromatic polynomial} $\chi(\spe{g}, k)$ counts the number of functions $f: N \to [k]$ subject to:
\begin{enumerate}
    \item For every $uv \in U$, we have $f(u) \neq f(v)$.
    \item For every $(u,v) \in D,$ we have $f(u) \leq f(v)$.
\end{enumerate}
The \emph{strong} chromatic polynomial $\bar{\chi}(\spe{g}, k)$ counts similar functions, only with strict inequalities for the second condition instead of the weak inequality. We introduced quasisymmetric function generalizations of both polynomials in \cite{white-1}, by showing how both polynomial invariants come from characters on $\spe{MG}$. 

Given a mixed graph $\spe{g}$, we define 
\[\bar{\chi}(\spe{g}) = 
\begin{cases} 
0 & \text{if $\spe{d}$ has at least one undirected edge} \\ 
1 & \text{otherwise.} 
\end{cases} \]

We proved the following in \cite{white-1}:
\begin{theorem}
The pairs $(\spe{MG}, \bar{\chi})$ and $(\spe{MG}, \chi)$ are linearized combinatorial Hopf monoids. Moreover, both $\bar{\chi}$ and $\chi$ are balanced convex characters.
\end{theorem}
We also showed that a function $f$ was $\chi$-proper for $\spe{g}$ if and only if $f$ is a strong coloring, and $f$ is $\bar{\chi}$-proper if and only if $f$ is a weak coloring.

We are able to apply Theorem \ref{thm:bigmain} to $\spe{MG}$ with respect to either character. We focus on describing the resulting quasisymmetric class functions. Let $\spe{g}$ be a mixed graph on a vertex set $N$. Then $\mathfrak{g} \in \mathfrak{S}_N$ is an \emph{automorphism} of $\spe{g}$ if and only if the following two conditions are satisfied:
\begin{enumerate}
    \item For each undirected edge $uv \in \spe{g}$, we have $\mathfrak{g}u\mathfrak{g}v \in \spe{g}$.
    \item For each directed edge $(u,v) \in \spe{g}$, we have $(\mathfrak{g}u, \mathfrak{g}v) \in \spe{g}$.
\end{enumerate}
Let $\Aut(\spe{g})$ be the automorphism group, and let $\mathfrak{g} \in \Aut(\spe{g})$.
Given a (weak) coloring $f: N \to \mathbb{N}$, we let $\mathfrak{g}f = f \circ \mathfrak{g}^{-1}$. This defines an action of $\Aut(\spe{g})$ on the set of (weak) colorings.

For $\mathfrak{g} \in \mathfrak{G}$, we see that \[\Psi_{\spe{MG}, \chi}(\spe{g}, \mathfrak{G}, \mathbf{x}; \mathfrak{g}) = \sum_{f: \mathfrak{g}f = f} \mathbf{x}^f \] where the sum is over colorings that are fixed by $\mathfrak{g}$.

Similarly, we see that \[\Psi_{\spe{MG}, \bar{\chi}}(\spe{g}, \mathfrak{G}, \mathbf{x}; \mathfrak{g}) = \sum_{f: \mathfrak{g}f = f} \mathbf{x}^f \] where the sum is over weak colorings that are fixed by $\mathfrak{g}$.
By Theorem \ref{thm:bigmain}, both of these invariants are $M$-increasing.

As an example, if we let $\spe{g}$ be the graph in Figure \ref{fig:mixedgraph}. Let $\mathbb{Z}/2\mathbb{Z}$ act by swapping $a$ with $c$ and $b$ with $d$. 
We have that \[\Psi_{\spe{MG}, \chi}(\spe{g}, \mathbb{Z}/2\mathbb{Z}, \mathbf{x}) = M_{2,2} + \rho(M_{1,1,2} + M_{2,1,1} + 3M_{1,1,1,1}) \]
and 
\begin{align*} \Psi_{\spe{MG}, \bar{\chi}}(\spe{g}, \mathbb{Z}/2\mathbb{Z}, \mathbf{x}) & = \rho M_{2,2} + \rho(M_{1,1,2}+M_{2,1,1}) +\Psi_{\spe{MG}, \chi}(\spe{g}, \mathbb{Z}/2\mathbb{Z}, \mathbf{x}). \end{align*}
\begin{figure}
\begin{center}
\begin{tikzpicture}
  \node[circle, draw=black, fill=white] (b) at (0,2) {$b$};
  \node[circle, draw=black, fill=white] (a) at (0,0) {$a$};
  \node[circle, draw=black, fill=white] (c) at (1.5,0) {$c$};
  \node[circle, draw=black, fill=white] (d) at (1.5,2) {$d$};
 \draw[-Latex] (b) -- (a);
  \draw[-Latex] (d) edge (c);
  \draw[-] (b) -- (c);
    \draw[-] (d) -- (a);

\end{tikzpicture}

\end{center}
\label{fig:mixedgraph}
\caption{an acyclic mixed graph.}
\end{figure}
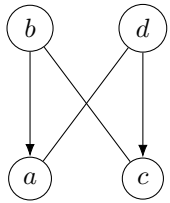

\subsection{Double Posets}
Now we will discuss double posets. Given a finite set $N$, a double poset is a triple $(N, \leq_1, \leq_2)$ where $\leq_1$ and $\leq_2$ are both partial orders on $N$. If we let $\spe{DP}[N]$ denote the set of double posets on $N$, then $\spe{DP}$ is a species. The Hopf algebra of double posets was introduced by Malvenuto and Reutenauer \cite{malvenuto}. We showed in \cite{white-1} that double posets also form a linearized combinatorial Hopf monoid $\spe{DP}$. The associated quasisymmetric function is a generalization of Gessel's $P$-partition enumerator, along with a generalization for labeled posets $(P, \omega)$. This quasisymmetric function is studied extensively by Grinberg \cite{grinberg}, who shows $F$-positivity results when the double poset is tertispecial. We introduced the corresponding quasisymmetric class function in \cite{white-2}. Now we show how the same function comes from linearized combinatorial Hopf monoids.

Given a double poset $\spe{d}$, a pair $(m,m') \in M$ is an \emph{inversion} if $m <_1 m'$ and $m' <_2 m$.
Given a double poset $\spe{d}$, we define 
\[\varphi(\spe{d}) = 
\begin{cases} 
0 & \text{if $\spe{d}$ has an inversion} \\ 
1 & \text{otherwise.} 
\end{cases} \]
We proved the following in \cite{white-1}.
\begin{theorem}
The pair $(\spe{DP}, \varphi)$ is a linearized combinatorial Hopf monoid. Moreover, $\varphi$ is a balanced convex character.
\end{theorem}
We also showed that the corresponding $\varphi$-proper colorings were double poset partitions.
A double poset partition is a function $\sigma: N \to \mathbb{N}$ subject to:
\begin{enumerate}
    \item for $x, y \in N$, if $x \leq_1 y$, the $\sigma(x) \leq \sigma(y)$.
    \item for $x, y \in N$, if $x \leq_1 y$ and $y <_2 x$, then $\sigma(x) < \sigma(y)$.
\end{enumerate}

Given a double poset $\spe{d}$ on $N$, a permutation $\mathfrak{g} \in \mathfrak{S}_N$ is an autmorphism if and only if it is an automorphism of both $\leq_1$ and $\leq_2$. Let $\Aut(\spe{d})$ be the automorphism group of $\spe{d}$. Then $\Aut(\spe{d})$ acts on the set of double poset partitions via $\mathfrak{g}\sigma = \sigma \circ \mathfrak{g}^{-1}$, where $\sigma$ is a double poset partition, and $\mathfrak{g} \in \Aut(\spe{d}).$

Let $\spe{d}$ be a double poset on $N$, and $\mathfrak{G} \subseteq \Aut(\spe{d})$. For $\mathfrak{g} \in \mathfrak{G}$, we have \[\Psi_{\spe{DP}, \varphi}(\spe{d}, \mathfrak{G}, \mathbf{x}; \mathfrak{g}) = \sum_{\sigma: \mathfrak{g}\sigma = \sigma} \mathbf{x}^{\sigma} \]
where the sum is over double poset partitions $\sigma$ that are fixed by $\mathfrak{g}$.

As an example, if we let $\spe{d}$ be the double poset in Figure \ref{fig:doubleposet}, where the Hasse diagram on the left is for $\leq_1$ and the Hasse diagram on the right is for $\leq_2$. We let $\mathbb{Z}/2\mathbb{Z}$ act by swapping $a$ with $c$ and $b$ with $d$. Then 
\begin{align*} \Psi_{\spe{DP}, \varphi}(\spe{d}, \mathbb{Z}/2\mathbb{Z}, \mathbf{x}) & = M_{2,2} + \rho(M_{1,1,2} + M_{1,2,1}+M_{2,1,1} + 2M_{1,1,1,1}). \end{align*} 
Our Theorem \ref{thm:bigmain} applies. In this case, several of the results were previously obtained in \cite{white-2}. 

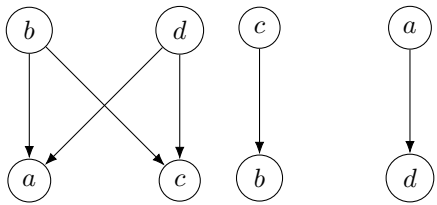
\begin{figure}
\begin{center}
\begin{tabular}{cc}
\begin{tikzpicture}
  \node[circle, draw=black, fill=white] (b) at (0,2) {$b$};
  \node[circle, draw=black, fill=white] (a) at (0,0) {$a$};
  \node[circle, draw=black, fill=white] (c) at (2,0) {$c$};
  \node[circle, draw=black, fill=white] (d) at (2,2) {$d$};
 \draw[-Latex] (b) -- (a);
  \draw[-Latex] (d) edge (c);
  \draw[-Latex] (b) -- (c);
    \draw[-Latex] (d) -- (a);

\end{tikzpicture}

& 

\begin{tikzpicture}
  \node[circle, draw=black, fill=white] (b) at (0,2) {$c$};
  \node[circle, draw=black, fill=white] (a) at (0,0) {$b$};
  \node[circle, draw=black, fill=white] (c) at (2,0) {$d$};
  \node[circle, draw=black, fill=white] (d) at (2,2) {$a$};
 \draw[-Latex] (b) -- (a);
  \draw[-Latex] (d) edge (c);

\end{tikzpicture}

\end{tabular}
\end{center}
\label{fig:doubleposet}
\caption{A double poset.}
\end{figure}

\subsection{Hopf monoid of generalized permutohedra}
 A \emph{generalized permutohedron} is a polytope whose normal fan is a coarsening of the braid arrangement. If we let $\spe{GP}[N]$ denote the set of all generalized permutohedra in $\mathbb{R}^N$, then we obtain the species of generalized permutohedra. Aguiar and Ardila \cite{aguiar-ardila} studied the linearized Hopf monoid $\spe{GP}$ of generalized permutohedra. This example contains many other Hopf monoids, such as graphs, posets, and matroids, as Hopf submonoids.

 The multiplication operation for $\spe{GP}$ consists of cartesian products of polytopes. As such, a generalized permutohedron is totally decomposable if and only if it is a cartesian product of zero-dimensional generalized permutohedra. In particular, it consists of only one vertex. Thus $\chi(\spe{p}) = 1$ if and only if $\spe{p}$ consists of only one vertex. This is the  basic character introduced by Aguiar and Ardila \cite{aguiar-ardila}.
They showed that, given a generalized permutohedron $\spe{p}$, a function $f: N \to \mathbb{N}$ is $\chi$-proper if it is maximized by a unique vertex of $\spe{p}$. We call these $\spe{p}$-generic functions.

Given $\mathfrak{g} \in \mathfrak{S}_N$, we can extend $\mathfrak{g}$ by linearity to get a linear isomorphism $\mathfrak{g}: \mathbb{R}^N \to \mathbb{R}^N$. We say that $\mathfrak{g}$ is an automorphism of $\spe{p}$ if $\mathfrak{g}(\spe{p}) = \spe{p}$.
Let $\Aut(\spe{p})$ be the automorphism group of $\spe{p}$. Then $\Aut(\spe{p})$ acts on the set of $\spe{p}$-generic functions. Given $\mathfrak{G} \subseteq \Aut(\spe{p})$ and $\mathfrak{g} \in \mathfrak{G}$, we see that the $\chi$-chromatic quasisymmetric class function of $\spe{p}$ is given by:
\[\Psi_{\spe{GP}, \chi}(\spe{p}, \mathfrak{G}, \mathbf{x}; \mathfrak{g}) = \sum_{f: \mathfrak{g}f = f} \mathbf{x}^f\] where the sum is over $\spe{p}$-generic functions that are fixed by $\mathfrak{g}$.

As an example, for $1 \leq i \leq j \leq n$, let $\Sigma_{i,j}$ be the convex hull of the standard basis vectors $e_i, e_{i+1}, \ldots, e_j$. Then Postnikov's description of the Loday realization of the associahedron is the Minkowski sum $A_n = \sum_{i \leq j} \Sigma_{i,j}$. Consider the permutation $\omega: [n] \to [n]$ given by $\omega(i) = n-i+1$. We see that, after we extend $\omega$ by linearity, $\omega(\Sigma_{i,j}) = \Sigma_{n-j+1, n-i+1}$. Hence $\omega(A_n) = A_n$, and we have an action of $\mathbb{Z}/2\mathbb{Z}$ on $A_n$.
 Then \[\Psi_{\spe{G}, \chi}(\spe{a}_3, \mathbb{Z}/2\mathbb{Z}, \mathbf{x}) = M_{2,1} + 3\rho M_{1,1,1}.\]
Since $\chi$ is a balanced convex character, Theorem \ref{thm:bigmain} applies to $\Psi_{\spe{GP}, \chi}(\spe{p}, \mathfrak{G}, \mathbf{x})$.


\subsection{Hypergraphs}
As another example, Aguiar and Ardila \cite{aguiar-ardila} also consider various Hopf monoid structures related to hypergraphs $\spe{HG}$. The corresponding polynomial invariants $\Psi_{\spe{HG}, \chi}(\spe{h}, \{e\}, x)$ were studied in greater detail in \cite{qsym-hypergraph}. A \emph{hypergraph} consists of a pair $(V, E)$, where $E$ is a multiset of nonempty subsets of $V$. Given a hypergraph $\spe{h}$, and edge $e$, and a function $f:N \to \mathbb{N}$, we say $v \in e$ is a \emph{local maximum} if $f(v) \geq f(u)$ for all $u \in e$.
A $\chi$-proper coloring is a function $f: V \to \mathbb{N}$ such that every edge $e$ has a unique local maximum.

Given a hypergraph $\spe{h}$, we see that $\Aut(\spe{h})$ consists of those $\mathfrak{g} \in \mathfrak{S}_N$ such that $\mathfrak{g}e \in E$ for all $e \in E$ (preserving multiplicity as well). Given $\mathfrak{G} \subseteq \Aut(\spe{h})$, we consider the action of $\mathfrak{G}$ on the set of $\chi$-proper colorings. 

For $\mathfrak{g} \in \mathfrak{G}$, we have
\[\Psi_{\spe{HG}, \chi}(\spe{h}, \mathfrak{G}, \mathbf{x}; \mathfrak{g}) = \sum_{f: \mathfrak{g}f = f} \mathbf{x}^f\] where the sum is over $\chi$-proper colorings of $\spe{h}$ that are fixed by $\mathfrak{g}$.

For example, let $\spe{h}$ be the hypergraph on $\{a,b,c,d\}$ whose edges consist of every subset of $\{a,b,c,d\}$ of size $3$, all with multiplicity one. Then 
\[\Psi_{\spe{HG}, \chi}(\spe{h}, \mathbb{Z}/4\mathbb{Z}, \mathbf{x}) = 3\rho M_{2,1,1}+6\rho M_{1,1,1,1}. \]
Since $\chi$ is a balanced convex character, Theorem \ref{thm:bigmain} applies to $\Psi_{\spe{HG}, \chi}(\spe{h}, \mathfrak{G}, \mathbf{x})$.

\subsection{Simplicial complexes}
Finally, another Hopf submonoid $\spe{SC}$ is related to simplicial complexes. Previously, Benedetti, Hallam, and Machacek \cite{benedetti-et-al} studied a Hopf algebra of simplicial complexes $SC$.  They also study a family of characters $\varphi_s$, indexed by positive integers.
The character $\varphi_s(\Sigma)$ is $1$ if $\dim(\Sigma) \leq s$ and $0$ otherwise. The Hopf monoid of simplicial complexes was studied by Aguiar and Ardila \cite{aguiar-ardila}.

It is not hard to show that $\varphi_s$ is a balanced convex character. The first two conditions in the definition of balanced convex character are easy to verify.
For a simplicial complex $\Sigma$ on $N$, and $S \subseteq N$, we have $\Sigma|_S$ is the induced subcomplex on $S$, while $\Sigma/S = \Sigma|_{N \setminus S}$. If $\varphi_s(\Sigma) = 1$, then $\Sigma$ contains no face on $s+1$ vertices. Clearly no induced subcomplex of $\Sigma$ does either. Thus, $\varphi_s$ satisfies the third condition to be a balanced convex character.

Let $\Sigma$ be a simplicial complex on $N$.
Then a $\varphi_s$-proper coloring is a function $f:N \to \mathbb{N}$ such that there does not exist $\sigma \in \Sigma$ nor $i \in \mathbb{N}$ such that $|\sigma| \geq s$ and $\sigma \subseteq f^{-1}(i)$. Let $F_s(\Sigma)$ be the collection of $\varphi_s$-proper colorings. 

Let $\Sigma$ be a simplicial complex on $N$. An automorphism of $\Sigma$ is a permutation $\mathfrak{g}$ with the property that $\mathfrak{g}\sigma \in \Sigma$ for all $\sigma \in \Sigma$. Let $\Aut(\Sigma)$ denote the automorphism group of $\Sigma$, and let $\mathfrak{G} \subseteq \Aut(\Sigma)$. For $\mathfrak{g} \in \mathfrak{G}$, the $\varphi_s$-chromatic quasisymmetric class function is given by
\[\Psi_{\spe{SC}, \varphi_s}(\Sigma, \mathfrak{G}, \mathbf{x}) = \sum_{f \in \Fix_{\mathfrak{g}}(F_s(\Sigma))} \mathbf{x}^f. \]

For example, let $\Sigma$ be the full simplex on $N = \{0,1, 2, 3\}$. Let $\mathbb{Z}/4\mathbb{Z}$ act on $N$ by the left regular action. Then $\mathbb{Z}/4\mathbb{Z}$ acts on $\Sigma$. Thus, \[\Psi_{\spe{SC}, \psi_2}(\Sigma, \mathfrak{G}, \mathbf{x}) = (\rho+1+\sgn)M_{2,2}+3\rho(M_{2,1,1}+M_{1,2,1}+M_{1,1,2}+2M_{1,1,1,1}).\]

\section{Balanced Relative Simplicial Complexes}
\label{sec:geometry}
Now we discuss balanced relative simplicial complexes, and their flag quasisymmetric class functions. This section is likely of independent interest.

\begin{figure}
\begin{center}
\begin{tikzpicture}
\draw[color=white, fill=gray!20] (30:2cm) -- (330:2cm) -- (210:2cm) -- (150:2cm) -- cycle;
  \node[circle, draw=red, fill=white, dashed, thick] (b) at (150:2cm) {$abc$};
  \node[circle, draw=red, fill=white, dashed, thick] (a) at (210:2cm) {$c$};
  \node[circle, draw=red, fill=white, dashed, thick] (e) at (330:2cm) {$acd$};
  \node[circle, draw=red, fill=white,dashed, thick] (d) at (30:2cm) {$a$};
  \node[circle, draw=black, fill=white,thick] (or) at (0:0cm) {$ac$};
  \draw[dashed, red, thick] (a) -- (b);
  \draw[dashed, red, thick] (d) -- (e);

  \draw[dashed, red, thick] (b) -- (d);
  \draw[dashed, red, thick] (a) -- (e);
  \draw[thick] (b) -- (or) -- (d);
  \draw[thick] (a) -- (or) -- (e);

\end{tikzpicture}
\end{center}
    \caption{A coloring complex $\Sigma$. Dashed lines correspond to faces that are not in $\Sigma$.}
    \label{fig:coloringcomplex}
\end{figure}
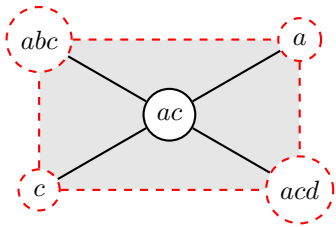

\begin{definition}
A \emph{balanced relative simplicial complex} of dimension $d$ on a vertex set $V$ is a non-empty collection $\Phi$ of subsets of $V$, along with a function $\kappa: V \to [d+1]$ with the following properties:
\begin{enumerate}
    \item For every $\rho \subseteq \sigma \subseteq \tau$, if $\rho, \tau \in \Phi$, then $\sigma \in \Phi$.
    \item For every $\rho \in \Phi$, there exists $\sigma \in \Phi$ such that $\rho \subseteq \sigma$ and $|\sigma| = d+1$,
    \item For every $\rho \in \Phi$, we have $\kappa(\rho) := \{\kappa(v): v \in \rho \}$ has size $|\rho|$.
\end{enumerate}
\end{definition}
The name comes from the fact that there exists simplicial complexes
$(\Gamma, \Sigma)$ with $\Gamma \subseteq \Sigma$, and $\Phi = \Sigma \setminus \Gamma$. Given $\sigma \in \Phi$, we let $\kappa(\sigma) = \{\kappa(v): v \in \sigma \}$.

 Given $S \subseteq [d+1]$, we let $f_S(\Phi)$ denote the number of faces $\sigma$ such that $\kappa(\sigma) = S$. This is the flag $f$-vector of $\Phi$. We encode the flag $f$-vector with a quasisymmetric function 
\[ \Hilb(\Phi, \mathbf{x})= \sum_{S \subseteq [d+1]} f_S(\Phi) M_{\alpha(S)}\]
where if $S = \{s_1, \ldots, s_k \}$, then $\alpha(S) = (s_1, s_2-s_1, s_3-s_2, \ldots, s_{k}-s_{k-1}, d+2-s_k)$. This is the \emph{flag quasisymmetric function} associated to $\Phi$. It has degree $d+2$.

Let $V(\Phi)$ be the vertex set of $\Phi$. A bijection $\mathfrak{g}: V \to V$ is an \emph{automorphism} of $\Phi$ if it satisfies the following two properties:
\begin{enumerate}
    \item For every $v \in V$, we have $\kappa(\mathfrak{g}v) = \kappa(v)$.
    \item For every $\{v_1, \ldots, v_k \} \in \Phi$, we have $\{\mathfrak{g}(v_1), \ldots, \mathfrak{g}(v_k) \} \in \Phi$.
\end{enumerate}

Let $\Aut(\Phi)$ be the group of automorphisms of $\Phi$, and fix a subgroup $\mathfrak{G} \subseteq \Aut(\Phi)$. For $\alpha \models |N|$, let $F_{\alpha}(\Phi) = \{\sigma \in \Phi: \alpha(\kappa(\sigma)) = \alpha \}$. Then $\mathfrak{G}$ acts on $F_{\alpha}(\Phi)$. We define the \emph{flag quasisymmetric class function} of $(\Phi, \mathfrak{G})$ to be 
\[ \Hilb(\Phi, \mathfrak{G}, \mathbf{x}) = \sum_{\sigma \in \Phi: \mathfrak{g}\sigma = \sigma}  M_{\alpha(\kappa(\sigma))}. \]
As an example, consider the balanced relative simplicial complex $\Phi$ with vertex set $2^{\{a,b,c,d\} } \setminus \{\emptyset, \{a,b,c,d\}\}$ appearing in Figure \ref{fig:coloringcomplex}. We denote the vertex $\{a,b,c\}$ as $abc$ for simplicity. We also let $\kappa(S) = |S|$. Then $\Phi$ is a balanced relative simplicial complex. We see that $\Aut(\Phi)$ is isomorphic to $\mathbb{Z}/2\mathbb{Z}$. Then \[F(\Phi, \mathbb{Z}/2\mathbb{Z}, \mathbf{x}) = M_{2,2}+\rho(M_{1,1,2}+M_{2,1,1}+2M_{1,1,1,1}).\]
There are natural orbital and polynomial specializations of $\Hilb(\Phi, \mathfrak{G}, \mathbf{x})$. The discussion is similar to what we already have for $\Psi_{\spe{H}, \varphi}(\spe{h}, \mathfrak{G}, \mathbf{x})$.

\begin{theorem}
Let $\Phi$ be a balanced relative simplicial complex of dimension $d$, and let $\mathfrak{G} \subseteq \Aut(\Phi)$. Then $\Hilb(\Phi, \mathfrak{G}, \mathbf{x})$ is a quasisymmetric class function, and is $M$-increasing.
\label{thm:increasing}
\end{theorem}
We prove a related proposition before giving a proof of the theorem.
Fix a balanced relative simplicial complex $\Phi$ of dimension $d$.
Given $\alpha \models d+2$, we let $V_{\alpha}(\Phi)$ be the span of $F_{\alpha}(\Phi)$, which is a $\mathfrak{G}$-module.
Given $\alpha \leq \beta \models d+2$, we define maps $\theta_{\alpha, \beta}: V_{\alpha}(\Phi) \to V_{\beta}(\Phi)$.
Given a face $\sigma \in \Phi$ of type $\alpha$, we let \begin{equation} \theta_{\alpha, \beta}(\sigma) = \sum\limits_{\tau \in F_{\beta}(\Phi): \sigma \subseteq \tau} \tau. \label{eq:theta} \end{equation}
\begin{proposition}
Let $\Phi$ be a balanced relative simplicial complex of dimension $d$, and let $\mathfrak{G} \subseteq \Aut(\Phi)$. Let $\alpha \leq \beta \leq \gamma $ be integer compositions of $|N|$. Then $\theta_{\alpha, \beta}$ is an injective $\mathfrak{G}$-invariant map. 
\label{prop:thetaprops}
\end{proposition}

Let us fix some notation. Let $\sigma \in \Phi$, and $S = \kappa(\sigma)$. For $T \subseteq S$, we define $\sigma|_T = \{v \in \sigma: \kappa(v) \in T \}$.
\begin{proof}
We see that $\theta_{\alpha, \beta}$ is $\mathfrak{G}$-invariant. To see that the map is injective, let $\vec{v} \in V_{\alpha}(\Phi)$ such that $\varphi_{\alpha, \beta}(\vec{v}) = \vec{0}$. Consider $\sigma \in F_{\alpha}(\Phi)$. Since $\Phi$ is a balanced relative simplicial complex, there exists a facet $\gamma$ such that $\sigma \subseteq \gamma$. Let $\tau \subseteq \gamma$ be such that $\alpha(\kappa(\tau)) = \beta$. We see that $\sigma \subseteq \tau \subseteq \gamma$. Since $\Phi$ is a relative simplicial complex, then $\tau \in F_{\beta}(\Phi)$. By definition, we see that \[[\tau]\theta_{\alpha, \beta}(\vec{v}) = \sum_{\pi \in F_{\alpha}(\Phi): \pi \subseteq \tau} [\pi] \vec{v}.\]
Let $S = \kappa(\sigma)$. Then $\alpha(S) = \alpha$. Moreover, if $\pi \in F_{\alpha}(\Phi)$ such that $\pi \subseteq \tau$, then $\pi = \tau|_S = \sigma$. Thus $[\tau] \varphi_{\alpha, \beta}(\vec{v}) = [\sigma] \vec{v}$. Since $[\tau]\varphi_{\alpha,\beta}(\vec{v}) = 0$, it follows that $[\sigma] \vec{v} = 0$. Since $\sigma \in F_{\alpha}(\Phi)$ was an arbitrary element of $F_{\alpha}(\Phi)$, and $F_{\alpha}(\Phi)$ is the basis of $V_{\alpha}(\Phi)$, it follows that $\vec{v} = \vec{0}$, and $\varphi_{\alpha, \beta}$ is injective.

\end{proof}

\begin{proof}[Proof of Theorem \ref{thm:increasing}]
For the first part, let $\alpha \models d+2$. Given $\mathfrak{g} \in \mathfrak{G}$, we see that $[M_{\alpha}]\Hilb(\Phi, \mathfrak{G}, \mathbf{x}; \mathfrak{g}) = |\Fix_{\mathfrak{g}}(F_{\alpha}(\Phi))|$. If we let $\chi_{\alpha}$ be the character that comes from the action of $\mathfrak{G}$ on $V_{\alpha}(\Phi)$, then we see that $[M_{\alpha}]\Hilb(\Phi, \mathfrak{G}, \mathbf{x}) = \chi_{\alpha}$. Thus $\Hilb(\Phi, \mathfrak{G}, \mathbf{x})$ is a quasisymmetric class function.

Let $\alpha \leq \beta \models d+2$. Since $\varphi_{\alpha, \beta}$ is injective and $\mathfrak{G}$-invariant, we see that $F_{\alpha}(\Phi)$ is isomorphic to a $\mathfrak{G}$-submodule of $V_{\beta}(\Phi)$. It follows that $\Hilb(\Phi, \mathfrak{G}, \mathbf{x})$ is $M$-increasing. 

\end{proof}

\subsection{The Coxeter Complex of type A}

In this section, we discuss the Coxeter complex of type $A$, since many balanced relative simplicial complexes we are interested in are subcomplexes of the Coxeter complex of type $A$. Our goal is to review the fact that faces of the Coxeter complex can be indexed by set compositions and by flags of subsets, and there is a $\mathfrak{S}_N$-invariant correspondence between both collections of indexing sets. 

Given a finite set $N$, let $V(N)$ be the collection of proper subsets of $N$. Then we define the Coxeter complex of the type $A$, $\Sigma_N$, as follows: 
\[\Sigma_N = \{ \{F_1, F_2, \ldots, F_k \}: \emptyset \subset F_1 \subset F_2 \subset \cdots \subset F_k \subset N \} \]
The faces of $\Sigma_N$ are \emph{flags} of proper subsets of $N$. We denote the faces $\{F_1, \ldots, F_k \}$ by $F_{\bdot}$ and given a flag $F_{\bdot} = F_1 \subset F_2 \subset \cdots \subset F_k$, we write $\ell(F_{\bdot}) = k+1$. We also define $\kappa(F) = \{|F_1|, \ldots, |F_k| \}$. Given $\mathfrak{g} \in \mathfrak{S}_N$, we let $\mathfrak{g}F_{\bdot} = \mathfrak{g}F_1 \subset \mathfrak{g}F_2 \subset \cdots \subset \mathfrak{g} F_k$. Thus we see that $\mathfrak{S}_N$ acts on $\Sigma_N$. We also see that $\kappa(\mathfrak{g}F_{\bdot}) = \kappa(F_{\bdot})$.

To every set composition $C \models N$, there is an associated flag $F(C) = \{F_1, \ldots, F_{\ell(C) - 1} \}$. We define $F_i = \bigcup_{j=1}^i C_j$. Note that $F(C) \in \Sigma_N$. Similarly, if $F_{\bdot} \in \Sigma_N$, and $\ell(F_{\bdot}) = k$, then there is an associated set composition $C(F)$, defined by:
\begin{enumerate}
    \item $C_1 = F_1$,
    \item $C_i = F_i \setminus F_{i-1}$ for $2 \leq i \leq k-1$, and
    \item $C_k = N \setminus F_{k-1}$.
\end{enumerate}

For example, to the set composition $13|2|45$, the associated edge in $\Sigma_{\{1,2,3,4,5 \}}$ is $\{1,3 \} \subset \{1,2,3 \}$. Hence, we can denote faces of the Coxeter complex by flags of subsets or by set compositions.

Now we show that the functions $F$ and $C$ are $\mathfrak{S}_N$-invariant.
Let $\mathfrak{g} \in \mathfrak{S}_N.$ Let $C \models N$. Then $F(\mathfrak{g}C)_i = \bigcup_{j=1}^i \mathfrak{g}(C_j) = \mathfrak{g} \left( \bigcup_{j=1}^i C_j \right) = \mathfrak{g}F(C)_i$. Thus $F(\mathfrak{g}C) = \mathfrak{g}F(C)$.

Now fix $F_{\bdot} \in \Sigma_N$ of length $k$. We see that $(\mathfrak{g}C(F_{\bdot}))_1 = \mathfrak{g}F_1 = \mathfrak{g}(C(F)_1)$. For $1 < i \leq k$, we see that $(\mathfrak{g}C(F_{\bdot}))_i = \mathfrak{g}(F_i \setminus F_{i-1}) = \mathfrak{g}(F_i) \setminus \mathfrak{g}(F_{i-1}) = \mathfrak{g}(C(F_{\bdot})_i)$. 
Finally, $(\mathfrak{g}(C(F_{\bdot}))_{k+1} = \mathfrak{g}(N \setminus F_k) = N \setminus \mathfrak{g}(F_k) = \mathfrak{g}(C(F_{\bdot})_{k+1})$. Thus $\mathfrak{g}C(F_{\bdot}) = C(\mathfrak{g}F_{\bdot})$.

\begin{proposition}
\label{prop:convex}

Let $N$ be a finite set. Then $C: \Sigma_N \to \spe{C}_N$ and $F: \spe{C}_N \to \Sigma_N$ are $\mathfrak{S}_N$-invariant bijections. 
\end{proposition}

\section{Geometric Realization}
\label{sec:realization}

We discuss how to associate balanced relative simplicial complexes to $\spe{H}$-structures. This association is only defined for linearized combinatorial Hopf monoids that have a \emph{balanced convex character}.
This construction was investigated in \cite{white-1}.

\begin{definition}
Let $(\spe{H},\varphi)$ be a linearized combinatorial Hopf monoid. Let $N$ be a finite set, and let $V$ be the collection of all proper subsets of $N$, and define $\kappa: V \to [|N|]$ by $\kappa(S) = |S|$. Given $\spe{h} \in \spe{H}[N]$, define the $\varphi$-coloring complex of $\spe{h}$ to be the subcomplex of $\Sigma_N$ given by 
\begin{align*} \Sigma_{\varphi}(\spe{h}) & = \{ F_{\bdot}: \varphi_{C(F_{\bdot})} (\spe{h}) = 1 \} \\ 
& = \{\{F_1, \ldots, F_k \}: F_1 \subset \cdots \subset F_k \mbox{ and } \prod_{i=1}^{k+1}\varphi(\spe{h}|_{F_i} / F_{i-1}) = 1 \}.\end{align*}
\label{def:coloringcomplex}
\end{definition}
The following Theorem was proven in \cite{white-1}:
\begin{theorem}
Let $(\spe{H},\varphi)$ be a linearized combinatorial Hopf monoid. Suppose that $\varphi$ is a balanced convex character. Let $N$ be a finite set, and let $V$ be the collection of all proper subsets of $N$, and define $\kappa: V \to [|N|]$ by $\kappa(S) = |S|$. Given $\spe{h} \in \spe{H}[N]$, let $\Sigma_{\varphi}(\spe{h})$ be the $\varphi$-coloring complex. Then $\Sigma_{\varphi}(\spe{h})$ is a balanced relative simplicial complex.
\label{lem:defsigma}
\end{theorem}

\begin{example}
Let $\spe{G}$ be the linearized combinatorial Hopf monoid of graphs. Then the character $\chi$ is balanced convex. Then the corresponding relative simplicial complex $\Sigma_{\chi}(\spe{g})$ consists of chains $S_1 \subset S_2 \subset \cdots \subset S_k \subset N$ such that $S_i \setminus S_{i-1}$ is an independent set for all $i$. One presentation for $\Phi(\spe{g})$ is as the pair $(\Sigma, \Gamma(\spe{g}))$, where $\Sigma$ is the Coxeter complex of type $A$, and $\Gamma(\spe{g})$ is collection of flags $S_1 \subset S_2 \subset \cdots \subset S_k \subset N$ where $S_{i+1} \setminus S_i$ must contain an edge for some $i$. The subcomplex $\Gamma(\spe{g})$ is the coloring complex, as introduced by Steingr\'imsson \cite{steingrimsson}.
\end{example}
\begin{example}
As another example, let $\spe{P}$ be the linearized Hopf monoid of posets. Let $\spe{p}$ be the poset in Figure \ref{fig:poset}. Then the $\chi$-coloring complex $\Sigma_{\chi}(\spe{p})$ is the balanced relative simplicial complex in Figure \ref{fig:coloringcomplex}.
\end{example}

\begin{theorem}
Let $(\spe{H},\varphi)$ be a linearized combinatorial Hopf monoid. Suppose that $\varphi$ is a balanced convex character. Let $N$ be a finite set.
Given $\spe{h} \in \spe{H}[N]$, let $\mathfrak{G} \subseteq \Aut(\spe{h})$. Then $\mathfrak{G} \subseteq \Aut(\Sigma_{\varphi}(\spe{h}))$, and
\[\Psi_{\spe{H}, \varphi}(\spe{h}, \mathfrak{G}, \mathbf{x}) = \Hilb(\Sigma_{\varphi}(\spe{h}), \mathfrak{G}, \mathbf{x}). \] 
\label{thm:autgeometry}
\end{theorem}
\begin{proof}
Let $N$ be a finite set, $\spe{h} \in \spe{H}[N]$, and $\mathfrak{G} \subseteq \Aut(\spe{h})$. Let $\mathfrak{g} \in \Aut(\spe{h})$ and $F_{\bdot} \in \Sigma_{\varphi}(\spe{h})$. Then $C(F_{\bdot}) \in \mathcal{C}_{\varphi}(\spe{h})$. In particular, $C(F_{\bdot}) \in \mathcal{C}_{\varphi, \alpha}(\spe{h})$ for some $\alpha$. 
Since $\mathfrak{G}$ acts on $\mathcal{C}_{\varphi, \alpha}(\spe{h})$, we see that $\mathfrak{g} C(F_{\bdot}) \in \mathcal{C}_{\varphi}(\spe{h})$. By Proposition \ref{prop:convex}, it follows that $\mathfrak{g}F_{\bdot} \in \Sigma_{\varphi}(\spe{h})$. Since $\mathfrak{g}C(F_{\bdot}) \in \mathcal{C}_{\varphi, \alpha}(\spe{h})$, we see that $\alpha(\kappa(\mathfrak{g}F_{\bdot})) = \alpha$. Thus $\mathfrak{G} \subseteq \Aut(\Sigma_{\varphi}(\spe{h})).$

By Equation \eqref{eq:fundamental}, we know that \[\Psi_{\spe{H}, \varphi}(\spe{h}, \mathfrak{G}, \mathbf{x}; \mathfrak{g}) = \sum_{C \in \mathcal{C}_{\varphi}(\spe{h}): \mathfrak{g}C = C} M_{\alpha(C)}.\] We also know that $\Sigma_{\varphi}(\spe{h}) = \{F(C): C \in \mathcal{C}_{\varphi}(\spe{h}) \}$, and that $M_{\alpha(C)} = M_{\alpha(F(C))}$. It follows from Proposition \ref{prop:convex} that $\mathfrak{g}C = C$ if and only if $\mathfrak{g}F(C) = F(C)$, so we have 
\[\sum_{C \in \mathcal{C}_{\varphi}(\spe{h}): \mathfrak{g}C = C} M_{\alpha(C)} = \sum_{F \in \Sigma_{\varphi}(\spe{h}): \mathfrak{g}F = F} M_{\alpha(F)}, \]
and the right-hand side is $\Hilb(\Sigma_{\varphi}(\spe{h}), \mathfrak{G}, \mathbf{x}; \mathfrak{g})$. 
\end{proof}

\begin{proof}[Proof of Theorem \ref{thm:bigmain}]
Let $\spe{H}$ be a linearized combinatorial Hopf monoid with balanced convex character $\varphi$. Let $N$ be a finite set, $\spe{h} \in \spe{H}[N]$, and $\mathfrak{G} \subseteq \Aut(\spe{h})$. Then by Theorem \ref{thm:autgeometry}, we know that $\Psi_{\spe{H}, \varphi}(\spe{h}, \mathfrak{G}, \mathbf{x}) = \Hilb(\Sigma_{\varphi}(\spe{h}), \mathfrak{G}, \mathbf{x})$. Since $\Sigma_{\varphi}(\spe{h})$ is a balanced relative simplicial complex, it follows from Theorem \ref{thm:increasing} that $\Psi_{\spe{H}, \varphi}(\spe{h}, \mathfrak{G}, \mathbf{x})$ is $M$-increasing. 

Let $1$ denote the trivial representation.
From Theorem \ref{thm:coloring2}, we see that $\Psi_{\spe{H}, \varphi}^O(\spe{h}, \mathfrak{G}, \mathbf{x}) = \langle 1, \Psi_{\spe{H}, \varphi}(\spe{h}, \mathfrak{G}, \mathbf{x}) \rangle.$ Let $\alpha \leq \beta$. Then by Proposition \ref{prop:global}, we see that $[M_{\alpha}]\Psi_{\spe{H}, \varphi}^O(\spe{h}, \mathfrak{G}, \mathbf{x}) \leq [M_{\beta}]\Psi_{\spe{H}, \varphi}^O(\spe{h}, \mathfrak{G}, \mathbf{x}).$ Thus $\Psi_{\spe{H}, \varphi}^O(\spe{h}, \mathfrak{G}, \mathbf{x})$ is also $M$-increasing.

From Theorem \ref{thm:coloring}, we see that $\Psi_{\spe{H}, \varphi}(\spe{h}, \mathfrak{G}, x) = \ps \Psi_{\spe{H}, \varphi}(\spe{h}, \mathfrak{G}, \mathbf{x})$. Since $\Psi_{\spe{H}, \varphi}(\spe{h}, \mathfrak{G}, \mathbf{x})$ is $M$-increasing, it follows from Proposition \ref{prop:global2} that the $f$-vector of $\Psi_{\spe{H}, \varphi}(\spe{h}, \mathfrak{G}, x)$ is effectively flawless. The fact that $(|N|-i)f_i(\Psi_{\spe{H}, \varphi}(\spe{h}, \mathfrak{G}, x)) \leq if_{i+1}(\Psi_{\spe{H}, \varphi}(\spe{h}, \mathfrak{G}, x))$ follows from Proposition \ref{prop:global3}.

We observe that
\begin{align*}\Psi^O_{\spe{H}, \varphi}(\spe{h}, \mathfrak{G}, x) & = \langle 1, \Psi_{\spe{H}, \varphi}(\spe{h}, \mathfrak{G}, x) \rangle  \\
& = \sum_{i=1}^{|N|} \langle 1, f_i(\Psi_{\spe{H}, \varphi}(\spe{h}, \mathfrak{G}, x)) \rangle \binom{x}{i} \end{align*}
where the first equality comes from Theorem \ref{thm:coloring2} and the second equality comes from Proposition \ref{prop:global}. 
Thus we see that \[f_i(\Psi^O_{\spe{H}, \varphi}(\spe{h}, \mathfrak{G}, x)) = \langle 1, f_i(\Psi_{\spe{H}, \varphi}(\spe{h}, \mathfrak{G}, x)) \rangle. \]

For the last set of inequalities, we need the following fact: if $\psi \leq_{\mathfrak{G}} \theta$, then $\langle 1, \psi \rangle \leq \langle 1, \theta \rangle$.
Suppose that we have constants $a$ and $b$, and $i \leq j$ such that \[af_i(\Psi_{\spe{H}, \varphi}(\spe{h}, \mathfrak{G}, x)) \leq_{\mathfrak{G}} bf_j(\Psi_{\spe{H}, \varphi}(\spe{h}, \mathfrak{G}, x)).\]
Then 
\begin{align*} af_i(\Psi^O_{\spe{H}, \varphi}(\spe{h}, \mathfrak{G}, x)) & = a\langle 1, f_i(\Psi_{\spe{H}, \varphi}(\spe{h}, \mathfrak{G}, x)) \rangle \\
& \leq b \langle 1, f_{j}(\Psi_{\spe{H}, \varphi}(\spe{h}, \mathfrak{G}, x)) \rangle \\
& = b f_j(\Psi^O_{\spe{H}, \varphi}(\spe{h}, \mathfrak{G}, x)).
\end{align*}
Thus  \[af_i(\Psi^O_{\spe{H}, \varphi}(\spe{h}, \mathfrak{G}, x)) \leq bf_j(\Psi^O_{\spe{H}, \varphi}(\spe{h}, \mathfrak{G}, x)).\]
In particular, for $i < |N|/2,$ if we set $a = b = 1$, and $j = i+1$, then we obtain 
\[ f_i(\Psi^O_{\spe{H}, \varphi}(\spe{h}, \mathfrak{G}, x)) \leq f_{i+1}(\Psi^O_{\spe{H}, \varphi}(\spe{h}, \mathfrak{G}, x)).\]
Similarly, if we set $a=b=1$ and $j=|N|-i$, then we obtain 
\[ f_i(\Psi^O_{\spe{H}, \varphi}(\spe{h}, \mathfrak{G}, x)) \leq f_{|N|-i}(\Psi^O_{\spe{H}, \varphi}(\spe{h}, \mathfrak{G}, x)).\]
Finally, for any $i$, if we set $a = |N|-i$, $b = i$, and $j=i+1$, we obtain 
\[ (|N|-i)f_i(\Psi^O_{\spe{H}, \varphi}(\spe{h}, \mathfrak{G}, x)) \leq if_{i+1}(\Psi^O_{\spe{H}, \varphi}(\spe{h}, \mathfrak{G}, x)).\]

\end{proof}




\bibliographystyle{amsplain}  
\bibliography{hopfclass}

\end{document}